\documentclass{amsart}

\usepackage{color,amssymb,graphicx,geometry}

\usepackage[colorlinks=true,urlcolor=blue, citecolor=red,linkcolor=blue,
linktocpage,pdfpagelabels, bookmarksnumbered,bookmarksopen]{hyperref}
\usepackage[hyperpageref]{backref}

\newtheorem{thm}{Theorem}[section]

\newtheorem{lm}[thm]{Lemma}
\newtheorem{prop}[thm]{Proposition}
\newtheorem{coro}[thm]{Corollary}

\theoremstyle{definition}

\newtheorem{rem}[thm]{Remark}
\newtheorem{exa}[thm]{Example}
\newtheorem*{ack}{Acknowledgments}

\numberwithin{equation}{section}

\title[Homogeneous Sobolev spaces]{A note on homogeneous Sobolev spaces\\ of fractional order}

\author[Brasco]{Lorenzo Brasco}

\address[L.\ Brasco]{Dipartimento di Matematica e Informatica
\newline\indent
Universit\`a degli Studi di Ferrara
\newline\indent
Via Machiavelli 35, 44121 Ferrara, Italy}
\email{lorenzo.brasco@unife.it}

\author[Salort]{Ariel Salort}

\address[A. Salort]{Departamento de Matem\'atica, FCEN
\newline\indent
Universidad de Buenos Aires and IMAS
\newline\indent
CONICET, Buenos Aires, Argentina}
\email{asalort@dm.uba.ar}


\begin{document}

\begin{abstract}
We consider a homogeneous fractional Sobolev space obtained by completion of the space of smooth test functions, with respect to a Sobolev--Slobodecki\u{\i} norm. We compare it to the fractional Sobolev space obtained by the $K-$method in real interpolation theory. We show that the two spaces do not always coincide and give some sufficient conditions on the open sets for this to happen. We also highlight some unnatural behaviors of the interpolation space. The treatment is as self-contained as possible.
\end{abstract}

\maketitle
\begin{center}
\begin{minipage}{8cm}
\small
\tableofcontents
\end{minipage}
\end{center}

\section{Introduction}

\subsection{Motivations}
In the recent years there has been a great surge of interest towards Sobolev spaces of fractional order. This is a very classical topic, essentially initiated by the Russian school in the 50s of the last century, with the main contributions given by Besov, Lizorkin, Nikol'ski\u{\i}, Slobodecki\u{\i} and their collaborators. Nowadays, we have a lot of monographies at our disposal on the subject. We just mention the books by Adams \cite{Ad, AF}, by Nikol'ski\u{\i} \cite{Ni} and those by Triebel \cite{Tr3, Tr2, Tr1}. We also refer the reader to \cite[Chapter 1]{Tr2} for an historical introduction to the subject.
\par
The reason for this revival lies in the fact that fractional Sobolev spaces seem to play a fundamental role in the study and description of a vast amount of phenomena, involving nonlocal effects. Phenomena of this type have a wide range of applications, we refer to \cite{BV} for an overview.
\par
There are many ways to introduce fractional derivatives and, consequently, Sobolev spaces of fractional order. Without any attempt of completeness, let us mention the two approaches which are of interest for our purposes:
\vskip.2cm
\begin{itemize}
\item a {\it concrete approach}, based on the introduction of explicit norms, which are modeled on the case of H\"older spaces. For example, by using the heuristic
\[
\delta_h^s u(x):=\frac{u(x+h)-u(x)}{|h|^s}\sim \mbox{ ``derivative of order $s$\,''},\qquad \mbox{ for }x,h\in\mathbb{R}^N, 
\]
a possible choice of norm is
\[
\left(\int \left\|\delta_h^s u\right\|_{L^p}^p\,\frac{dh}{|h|^N}\right)^\frac{1}{p}, 
\] 
and more generally
\[
\left(\int \left\|\delta_h^s u\right\|_{L^p}^q\,\frac{dh}{|h|^N}\right)^\frac{1}{q},\qquad \mbox{ for } 1\le q\le \infty.
\]
Observe that the integral contains the singular kernel $|h|^{-N}$, thus functions for which the norm above is finite must be better than just merely $s-$H\"older regular, in an averaged sense; 
\vskip.2cm
\item an {\it abstract approach}, based on the so-called {\it interpolation methods}. The foundations of these methods were established at the beginning of the 60s of the last century, by Calder\'on, Gagliardo, Krejn, Lions and Petree, among others. A comprehensive treatment of this approach can be found for instance in the books \cite{BL, BS88, T78} and references therein
\par
In a nutshell, the idea is to define a scale of ``intermediate spaces'' between $L^p$ and the standard Sobolev space $W^{1,p}$, by means of a general abstract construction.
The main advantage of this second approach is that many of the properties of the spaces constructed in this way can be extrapolated in a direct way from those of the two ``endpoint'' spaces $L^p$ and $W^{1,p}$.
\end{itemize}
As mentioned above, actually other approaches are possible: a possibility is to use the Fourier transform. Another particularly elegant approach consists in taking the convolution with a suitable kernel (for example, heat or Poisson kernels are typical choices) and looking at the rate of blow-up of selected $L^p$ norms with respect to the convolution parameter. However, we will not consider these constructions in the present paper, we refer the reader to \cite{Tr2} for a wide list of definitions of this type.
\par
In despite of the explosion of literature on Calculus of Variations settled in fractional Sobolev spaces of the last years, the abstract approach based on interpolation seems to have been completely neglected or, at least, overlooked. For example, the well-known survey paper \cite{DPV}, which eventually became a standard reference on the field, does not even mention interpolation techniques.
\subsection{Aims}
The main scope of this paper is to revitalize some interest towards interpolation theory in the context of fractional Sobolev spaces.
In doing this, we will resist the temptation of any unnecessary generalization. Rather, we will focus on a particular, yet meaningful, question which can be resumed as follows:
 \vskip.2cm
 \centerline{\it Given a concrete fractional Sobolev space} 
 \centerline{\it of functions vanishing ``at the boundary'' of a set,}
 \centerline{\it does it coincide with an interpolation space?}
 \vskip.2cm
 We can already anticipate the conclusions of the paper and say that {\it this is not always true}. Let us now try to enter more in the details of the present paper. 
 \par
Our  concerns involve the so-called \emph{homogeneous fractional Sobolev-Slobodecki\u{\i} spaces} $\mathcal{D}^{s,p}_0(\Omega)$. Given an open set $\Omega\subset \mathbb{R}^N$, an exponent $1\leq p <\infty$ and a parameter $0<s<1$, it is defined as the completion of $C_0^\infty(\Omega)$ with respect to the norm
$$ 
u\mapsto [u]_{W^{s,p}(\mathbb{R}^N)}:=\left(\iint_{\mathbb{R}^N\times\mathbb{R}^N} \frac{|u(x)-u(y)|^p}{|x-y|^{N+s\,p}}\,dx\,dy\right)^\frac{1}{p}.
$$
Such a space is the natural fractional counterpart of the homogeneous Sobolev space $\mathcal{D}^{1,p}_0(\Omega)$, defined 
as the completion of $C^\infty_0(\Omega)$ with respect to the norm
\[
u\mapsto \left(\int_\Omega |\nabla u|^p\,dx\right)^\frac{1}{p}.
\]
The space $\mathcal{D}^{1,p}_0(\Omega)$ has been first studied by Deny and Lions in \cite{DL}, among others. We recall that $\mathcal{D}^{1,p}_0(\Omega)$ is a natural setting for studying variational problems of the type
\[
\inf\left\{\frac{1}{p}\,\int_\Omega |\nabla u|^p\,dx-\int_\Omega f\,u\,dx \right\},
\] 
supplemented with Dirichlet boundary conditions, in absence of regularity assumptions on the boundary $\partial\Omega$. In the same way, the space $\mathcal{D}^{s,p}_0(\Omega)$ is the natural framework for studying minimization problems containing functionals of the type
\begin{equation}
\label{nonlocalF}
\frac{1}{p}\,\iint_{\mathbb{R}^N\times\mathbb{R}^N} \frac{|u(x)-u(y)|^p}{|x-y|^{N+s\,p}}\,dx\,dy-\int_\Omega f\,u\,dx,
\end{equation}
in presence of {\it nonlocal} Dirichlet boundary conditions, i.e. the values of $u$ are prescribed on the whole complement $\mathbb{R}^N\setminus\Omega$. Observe that even if this kind of boundary conditions may look weird, these are the correct ones when dealing with energies \eqref{nonlocalF}, which take into account interactions ``from infinity''.
\par
The connection between the two spaces $\mathcal{D}^{1,p}_0(\Omega)$ and $\mathcal{D}^{s,p}_0(\Omega)$ is better appreciated by recalling that for $u\in C^\infty_0(\Omega)$, we have  (see \cite{BBM} and \cite[Corollary 1.3]{Po})
\[
\lim_{s\nearrow 1} (1-s)\,\iint_{\mathbb{R}^N\times\mathbb{R}^N} \frac{|u(x)-u(y)|^p}{|x-y|^{N+s\,p}}\,dx\,dy=\alpha_{N,p}\,\int_{\Omega} |\nabla u|^p\,dx,
\]
with
\[
\alpha_{N,p}=\frac{1}{p}\,\int_{\mathbb{S}^{N-1}}|\langle \omega,\mathbf{e}_1\rangle|^p\,d\mathcal{H}^{N-1}(\omega),\qquad \mathbf{e}_1=(1,0,\dots,0). 
\]
On the other hand, as $s\searrow 0$ we have (see \cite[Theorem 3]{MS})
\[
\lim_{s\searrow 0} s\,\iint_{\mathbb{R}^N\times\mathbb{R}^N} \frac{|u(x)-u(y)|^p}{|x-y|^{N+s\,p}}\,dx\,dy=\beta_{N,p}\,\int_{\Omega}|u|^p\,dx,
\]
with
\[
\beta_{N,p}=\frac{2\,N\,\omega_N}{p},
\]
and $\omega_N$ is the volume of the $N-$dimensional unit ball. These two results reflect the ``interpolative'' nature of the space $\mathcal{D}^{s,p}_0(\Omega)$, which will be however discussed in more detail in the sequel.
\par
Indeed, one of our goals is to determine whether $\mathcal{D}^{s,p}_0(\Omega)$ coincides or not with the {\it real interpolation space} $\mathcal{X}^{s,p}_0(\Omega)$ defined as the completion of $C_0^\infty(\Omega)$ with respect to the norm
\[
\|u\|_{\mathcal{X}^{s,p}_0(\Omega)}:=\left(\int_0^{+\infty} \left(\frac{K(t,u,L^p(\Omega),\mathcal{D}^{1,p}_0(\Omega))}{t^s}\right)^p\,\frac{dt}{t}\right)^\frac{1}{p}.
\]
Here $K(t,\cdot,L^p(\Omega),\mathcal{D}^{1,p}_0(\Omega))$ is the $K-$functional associated to the spaces $L^p(\Omega)$ and $\mathcal{D}^{1,p}_0(\Omega)$, see Section \ref{sec:comparison} below for more details.
\par
In particular, we will be focused on obtaining double-sided norm inequalities leading to answer our initial question, i.e. estimates of the form
\[
\frac{1}{C}\,[u]_{W^{s,p}(\mathbb{R}^N)}\le \|u\|_{\mathcal{X}^{s,p}_0(\Omega)}\le C\,[u]_{W^{s,p}(\mathbb{R}^N)},\qquad u\in C^\infty_0(\Omega).
\]
Moreover, we compute carefully the dependence on the parameter $s$ of the constant $C$. Indeed, we will see that
$C$ can be taken independent of $s$.

\subsection{Results}

We now list the main achievements of our discussion:
\begin{itemize}
\item[1.] the space $\mathcal{D}^{s,p}_0(\Omega)$ is always larger than $\mathcal{X}^{s,p}_0(\Omega)$ (see Proposition \ref{lm:3}) and they do not coincide for general open sets, as we exhibit with an explicit example (see Example \ref{exa:counternorma}); 
\vskip.2cm
\item[2.] they actually coincide on a large class of domains, i.e. bounded convex sets (Theorem \ref{thm:convex}), convex cones (Corollary \ref{coro:cones}), Lipschitz sets (Theorem \ref{thm:finally});
\vskip.2cm
\item[3.] the Poincar\'e constants for the embeddings 
\[
\mathcal{D}^{s,p}_0(\Omega)\hookrightarrow L^p(\Omega) \qquad \mbox{ and }\qquad \mathcal{D}^{1,p}_0(\Omega)\hookrightarrow L^p(\Omega),
\]
are equivalent for the classes of sets at point 2 (Theorem \ref{prop:doubleside}). More precisely, by setting
\[
\lambda^s_p(\Omega)=\inf_{u\in C_0^\infty}\Big\{  [u]_{W^{s,p}(\Omega)}^p\ :\  \|u\|_{L^p(\Omega)}=1\Big\},\qquad 0<s<1,
\]
and
\[
\lambda^1_p(\Omega)=\inf_{u\in C_0^\infty}\left\{\int_\Omega |\nabla u|^p\,dx\, :\,  \|u\|_{L^p(\Omega)}=1\right\},
\]
we have 
\[
\frac{1}{C}\,\Big(\lambda^1_p(\Omega)\Big)^s\le s\,(1-s)\,\lambda^s_p(\Omega)\le C\, \Big(\lambda^1_p(\Omega)\Big)^s.
\]
Moreover, on convex sets the constant $C>0$ entering in the relevant estimate is universal, i.e. {\it it depends on $N$ and $p$ only}. On the other hand, we show that this equivalence fails if we drop any kind of regularity assumptions on the sets (see Remark \ref{rem:poinpoin}). 
\end{itemize}
As a byproduct of our discussion, we also highlight some weird and unnatural behaviors of the interpolation space $\mathcal{X}^{s,p}_0(\Omega)$:
\begin{itemize}
\item the ``extension by zero'' operator $\mathcal{X}^{s,p}_0(\Omega)\hookrightarrow \mathcal{X}_0^{s,p}(\mathbb{R}^N)$ is not continuous for general open sets (see Remark \ref{rem:nightmare}). This is in contrast with what happens for the spaces $L^p(\Omega)$, $\mathcal{D}^{1,p}_0(\Omega)$ and $\mathcal{D}^{s,p}_0(\Omega)$;
\vskip.2cm
\item the sharp Poincar\'e interpolation constant 
\[
\Lambda^s_p(\Omega)=\inf_{u\in C^\infty_0(\Omega)} \Big\{  \|u\|_{\mathcal{X}^{s,p}_0(\Omega)}^p\ :\  \|u\|_{L^p(\Omega)}=1\Big\},\qquad 0<s<1
\] 
is sensitive to removing sets with zero capacity. In other words, if we remove a compact set $E\Subset\Omega$ having zero capacity in the sense of $\mathcal{X}^{s,p}_0(\Omega)$, then (see Lemma \ref{lm:unnatural})
\[
\Lambda^s_p(\Omega\setminus E)>\Lambda^s_p(\Omega).
\] 
Again, this is in contrast with the case of $\mathcal{D}^{1,p}_0(\Omega)$ and $\mathcal{D}^{s,p}_0(\Omega)$.
\end{itemize}

\begin{rem}
As recalled at the beginning, nowadays there is a huge literature on Sobolev spaces of fractional order. Nevertheless, to the best of our knowledge, a detailed discussion on the space $\mathcal{D}^{s,p}_0(\Omega)$ in connection with interpolation theory seems to be missing. For this reason, we believe that our discussion is of independent interest. 
\par 
We also point out that for Sobolev spaces of functions not necessarily vanishing at the boundary, there is a very nice paper \cite{CHM} by Chandler-Wilde, Hewett and Moiola comparing ``concrete'' constructions with the interpolation one.
\end{rem}

\subsection{Plan of the paper}

In Section \ref{sec:prelim} we present the relevant Sobolev spaces, constructed with the concrete approach based on the so-called {\it Sobolev-Slobodecki\u{\i} norms}. Then in Section \ref{sec:comparison} we introduce the homogeneous interpolation space we want to work with. Essentially, no previous knowledge of interpolation theory is necessary. 
\par
The comparison between the concrete space and the interpolation one is contained in Section \ref{sec:concrete}. This in turn is divided in three subsections, each one dealing with a different class of open sets. We point out here that we preferred to treat convex sets separately from Lipschitz sets, for two reasons: the first one is that for convex sets the comparison between the two spaces can be done ``by hands'', without using any extension theorem.  This in turn permits to have a better control on the relevant constants entering in the estimates. The second one is that in proving the result for Lipschitz sets, we actually use the result for convex sets.
\par
In order to complement the comparison between the two spaces, in Section \ref{sec:capacities} we compare the two relevant notions of capacity, naturally associated with the norms of these spaces. Finally, Section \ref{sec:poincare} compares the Poincar\'e constants.
\par
The paper ends with 3 appendices: the first one contains the construction of a counter-example used throughout the whole paper; the second one proves a version of the one-dimensional Hardy inequality; the last one contains a geometric expedient result dealing with convex sets.

\begin{ack}
The first author would like to thank Yavar Kian and Antoine Lemenant for useful discussions on Stein's and Jones' extension theorems. Simon Chandler-Wilde is gratefully acknowledged for some explanations on his paper \cite{CHM}. 
 This work started during a visit of the second author to the University of Ferrara in October 2017. 
\end{ack}

\section{Preliminaries}
\label{sec:prelim}
\subsection{Basic notation}
In what follows, we will always denote by $N$ the dimension of the ambient space. For an open set $\Omega\subset\mathbb{R}^N$, we indicate by $|\Omega|$ its $N-$dimensional Lebesgue measure. The symbol $\mathcal{H}^k$ will stand for the $k-$dimensional Hausdorff measure. Finally, we set
\[
B_R(x_0)=\{x\in\mathbb{R}^N\, :\, |x-x_0|<R\},
\]
and
\[
\omega_N=|B_1(0)|.
\]
\subsection{Sobolev spaces} 
\label{sec.sobolev}
For $1\leq p<\infty$ and an open set $\Omega\subset \mathbb{R}^N$, we use the classical definition
$$
W^{1,p}(\Omega):=\left\{u\in L^p(\Omega)\,:\, \int_\Omega |\nabla u|^p\,dx <+\infty\right\}.
$$
This is a Banach space endowed with the norm
$$
\|u\|_{W^{1,p}(\Omega)}=\left(\|u\|^p_{L^p(\Omega)}+\|\nabla u\|^p_{L^p(\Omega)}\right)^\frac{1}{p}.
$$
We also denote by $\mathcal{D}^{1,p}_0(\Omega)$ the {\it homogeneous Sobolev space}, defined as the completion of $C^\infty_0(\Omega)$ with respect to the norm 
\[
u\mapsto\|\nabla u\|_{L^p(\Omega)}.
\]
If the open set $\Omega\subset\mathbb{R}^N$ supports the classical Poincar\'e inequality
\[
c\,\int_\Omega |u|^p\,dx\le \int_\Omega |\nabla u|^p\,dx,\qquad \mbox{ for every }u\in C^\infty_0(\Omega),
\]
then $\mathcal{D}^{1,p}_0(\Omega)$ is indeed a functional space and it coincides with the closure in $W^{1,p}(\Omega)$ of $C^\infty_0(\Omega)$. We will set
\[
\lambda^1_p(\Omega)=\inf_{u\in C^\infty_0(\Omega)} \Big\{\|\nabla u\|^p_{L^p(\Omega)}\, :\, \|u\|_{L^p(\Omega)}=1\Big\}.
\]
It occurs $\lambda^1_p(\Omega)=0$ whenever $\Omega$ does not support such a Poincar\'e inequality.
\begin{rem}
\label{rem:density}
We remark that one could also consider the space
$$
W^{1,p}_0(\Omega):= \{u\in W^{1,p}(\mathbb{R}^N)\colon u = 0 \text{ a.e. in } \mathbb{R}^N\setminus\Omega\}.
$$
It is easy to see that $\mathcal{D}^{1,p}_0(\Omega)\subset {W}^{1,p}_0(\Omega)$, whenever $\mathcal{D}^{1,p}_0(\Omega)\hookrightarrow L^p(\Omega)$. If in addition $\partial\Omega$ is continuous, then both spaces are known to coincide, thanks to the density of $C^\infty_0(\Omega)$ in $W^{1,p}_0(\Omega)$, see \cite[Theorem 1.4.2.2]{Gr}.
\end{rem}
 
\subsection{A homogeneous Sobolev--Slobodecki\u{\i} space} 
 
Given $0<s< 1$ and $1\leq p <\infty$, the fractional Sobolev space $W^{s,p}(\mathbb{R}^N)$ is defined as
$$
W^{s,p}(\mathbb{R}^N) := \left\{u\in L^p(\mathbb{R}^N)\,:\, [u]_{W^{s,p}(\mathbb{R}^N)}<+\infty \right\},
$$
where the {\it Sobolev--Slobodecki\u{\i} seminorm} $[\,\cdot\,]_{W^{s,p}(\mathbb{R}^N)}$ is defined as
$$
[u]_{W^{s,p}(\mathbb{R}^N)} := 
\left(\iint_{\mathbb{R}^N\times\mathbb{R}^N} \frac{|u(x)-u(y)|^p}{|x-y|^{N+s\,p}}\, dx\,dy\right)^\frac1p.
$$
This is a Banach space endowed with the norm
$$
\|u\|_{W^{s,p}(\mathbb{R}^N)} = \left(\|u\|_{L^p(\mathbb{R}^N)}^p + [u]_{W^{s,p}(\mathbb{R}^N)}^p\right)^\frac{1}{p}.
$$
In what follows, we need to consider nonlocal homogeneous Dirichlet boundary conditions, outside an open set $\Omega\subset\mathbb{R}^N$. In this setting, it is customary to consider the {\it homogeneous Sobolev--Slobodecki\u{\i} space} $\mathcal{D}^{s,p}_0(\Omega)$. The latter is defined as the completion of $C^\infty_0(\Omega)$ with respect to the norm 
\[
u\mapsto [u]_{W^{s,p}(\mathbb{R}^N)}.
\]
Observe that the latter is indeed a norm on $C^\infty_0(\Omega)$.
Whenever the open set $\Omega\subset\mathbb{R}^N$ admits the following Poincar\'e inequality
\[
c\,\int_\Omega |u|^p\,dx\le \iint_{\mathbb{R}^N\times\mathbb{R}^N} \frac{|u(x)-u(y)|^p}{|x-y|^{N+s\,p}}\, dx\,dy, \qquad \mbox{ for every }u\in C_0^\infty(\Omega),
\]
we get that $\mathcal{D}^{s,p}_0(\Omega)$ is a functional space continuously embedded in $L^p(\Omega)$. In this case, it coincides with the closure in $W^{s,p}(\mathbb{R}^N)$ of $C^\infty_0(\Omega)$. We endow the space $\mathcal{D}^{s,p}_0(\Omega)$ with the norm
\[
\|u\|_{\mathcal{D}^{s,p}_0(\Omega)}:=[u]_{W^{s,p}(\mathbb{R}^N)}.
\]
We also define 
\[
\lambda^s_p(\Omega)=\inf_{u\in C^\infty_0(\Omega)} \Big\{\|u\|^p_{\mathcal{D}^{s,p}_0(\Omega)}\, :\, \|u\|_{L^p(\Omega)}=1\Big\},
\]
i.e. this is the sharp constant in the relevant Poincar\'e inequality. Some embedding properties of the space $\mathcal{D}^{s,p}_0(\Omega)$ are investigated in \cite{Fra}.
\begin{rem}
As in the local case, one could also consider the space
$$
W^{s,p}_0(\Omega):= \{u\in W^{s,p}(\mathbb{R}^N)\colon u = 0 \text{ a.e. in } \mathbb{R}^N\setminus\Omega\}.
$$
It is easy to see that $\mathcal{D}^{s,p}_0(\Omega)\subset {W}^{s,p}_0(\Omega)$, whenever $\mathcal{D}^{s,p}_0(\Omega)\hookrightarrow L^p(\Omega)$. As before, whenever $\partial\Omega$ is continuous, then both spaces are known to coincide, again thanks to the density of $C^\infty_0(\Omega)$ in $W^{s,p}_0(\Omega)$, see \cite[Theorem 1.4.2.2]{Gr}.
\end{rem}


\subsection{Another space of functions vanishing at the boundary}
Another natural fractional Sobolev space of functions ``vanishing at the boundary'' is given by the completion of $C^\infty_0(\Omega)$ with respect to the localized norm
\[
[u]_{W^{s,p}(\Omega)}=\left(\iint_{\Omega\times\Omega} \frac{|u(x)-u(y)|^p}{|x-y|^{N+s\,p}}\,dx\,dy\right)^\frac{1}{p}.
\]
We will denote this space by $\mathring{D}^{s,p}(\Omega)$. We recall the following
\begin{lm}
Let $1<p<\infty$ and $0<s<1$. For every $\Omega\subset\mathbb{R}^N$ open bounded Lipschitz set, we have:
\begin{itemize}
\item if $s\,p>1$, then
\[
\mathcal{D}^{s,p}_0(\Omega)=\mathring{D}^{s,p}(\Omega);
\]
\item if $s\,p\le 1$, then there exists a sequence $\{u_n\}_{n\in\mathbb{N}}\subset C^\infty_0(\Omega)$ such that
\[
\lim_{n\to\infty} \frac{\|u_n\|_{\mathring{D}^{s,p}(\Omega)}}{\|u_n\|_{\mathcal{D}^{s,p}_0(\Omega)}}=0.
\]
\end{itemize}
\end{lm} 
\begin{proof}
The proof of the first fact is contained in \cite[Theorem B.1]{BLP}. 
\par
As for the case $s\,p\le 1$, in \cite[Section 2]{Dy} Dyda constructed a sequence $\{u_n\}_{n\in\mathbb{N}}\subset C^\infty_0(\Omega)$ such that
\[
\lim_{n\to\infty}\|u_n\|_{\mathring{D}^{s,p}(\Omega)}=0\qquad \mbox{ and }\qquad \lim_{n\to\infty}\|u_n-1_\Omega\|_{L^p(\Omega)}=0.
\]
By observing that for such a sequence we have 
\[
\lim_{n\to\infty}\|u_n\|_{\mathcal{D}^{s,p}_0(\Omega)}\ge \Big(\lambda^s_p(\Omega)\Big)^\frac{1}{p}\,\lim_{n\to\infty}\|u_n\|_{L^p(\Omega)}=\Big(\lambda^s_p(\Omega)\,|\Omega|\Big)^\frac{1}{p},
\]
we get the desired conclusion. In the inequality above, we used that $\lambda^s_p(\Omega)>0$ for an open bounded set, thanks to \cite[Corollary 5.2]{BC}.
\end{proof}

\begin{rem}
Clearly, we always have
\[
\|u\|_{\mathring{D}^{s,p}(\Omega)}\le \|u\|_{\mathcal{D}^{s,p}_0(\Omega)},\qquad \mbox{ for every } u\in C^\infty_0(\Omega).
\]
As observed in \cite{DV2}, the reverse inequality
\begin{equation}
\label{e0}
\|u\|_{\mathcal{D}^{s,p}_0(\Omega)}\le C\, \|u\|_{\mathring{D}^{s,p}(\Omega)},\qquad \mbox{ for every } u\in C^\infty_0(\Omega),
\end{equation}
is equivalent to the validity of the Hardy-type inequality
\[
\int_\Omega |u(x)|^p\left(\int_{\mathbb{R}^N\setminus \Omega} |x-y|^{-N-s\,p}\,y\right)\,dx\le C\,\iint_{\Omega\times\Omega} \frac{|u(x)-u(y)|^p}{|x-y|^{N+s\,p}}\,dx\,dy.
\]
A necessary and sufficient condition for this to happen is proved in \cite[Proposition 2]{DV2}. We also observe that the failure of \eqref{e0} implies that in general the ``extension by zero'' operator 
\[
\mathcal{T}_0:\mathring{D}^{s,p}(\Omega)\to \mathring{D}^{s,p}(\mathbb{R}^N),
\] 
is not continuous. We refer to \cite{DV2} for a detailed discussion of this issue.
\end{rem}
\begin{rem}
The space $\mathring{D}^{s,p}(\Omega)$ is quite problematic in general, especially in the case $s\,p\le 1$ where it may fail to be a functional space. A more robust variant of this space is
\[
\widetilde{D}^{s,p}(\Omega)=``\mbox{closure of $C^\infty_0(\Omega)$ in }W^{s,p}(\Omega)".
\]
By definition, this is automatically a functional space, continuously contained in $W^{s,p}(\Omega)$. It is a classical fact that if $\Omega$ is a bounded open set with smooth boundary, then
\[
\widetilde{D}^{s,p}(\Omega)=W^{s,p}(\Omega),\qquad \mbox{ for } s\,p< 1,
\] 
and
\[
\widetilde{D}^{s,p}(\Omega)=W^{s,p}_0(\Omega),\qquad \mbox{ for }s\,p>1,
\]
see Theorem \cite[Theorem 3.4.3]{Tr1}.
Moreover, we also have
\[
\widetilde{D}^{s,p}(\Omega)=\mathcal{D}^{s,p}_0(\Omega),\qquad \mbox{ for }s\,p\not =1,
\]
see for example \cite[Proposition B.1]{BLP}.
\end{rem} 
 
\section{An interpolation space}
\label{sec:comparison}

Let $\Omega\subset\mathbb{R}^N$ be an open set. If $X(\Omega)$ and $Y(\Omega)$ are two normed vector spaces containing $C^\infty_0(\Omega)$ as a dense subspace, we define for every $t>0$ and $u\in C^\infty_0(\Omega)$ the {\it $K-$functional}
\begin{equation}
\label{K}
K(t,u,X(\Omega),Y(\Omega)):=\inf_{v\in C^\infty_0(\Omega)} \Big\{\|u-v\|_{X(\Omega)}+t\,\|v\|_{Y(\Omega)}\Big\}.
\end{equation}
We are interested in the following specific case: let us take $0<s<1$ and $1<p<\infty$, we choose
\[
X(\Omega)=L^p(\Omega)\qquad \mbox{ and }\qquad Y(\Omega)=\mathcal{D}^{1,p}_0(\Omega).
\]
Then we use the notation
\[
\|u\|_{\mathcal{X}^{s,p}_0(\Omega)}:=\left(\int_0^{+\infty} \left(\frac{K(t,u,L^p(\Omega),\mathcal{D}^{1,p}_0(\Omega))}{t^s}\right)^p\,\frac{dt}{t}\right)^\frac{1}{p},\qquad u\in C^\infty_0(\Omega).
\]
It is standard to see that this is a norm on $C^\infty_0(\Omega)$, see \cite[Section 3.1]{BL}. We will indicate by $\mathcal{X}^{s,p}_0(\Omega)$ the completion of $C^\infty_0(\Omega)$ with respect to this norm.
\vskip.2cm
The first result is the Poincar\'e inequality for the interpolation space $\mathcal{X}^{s,p}_0(\Omega)$. The main focus is on the explicit dependence of the constant on the local Poincar\'e constant $\lambda^1_{p}$.
\begin{lm}
Let $1<p<\infty$ and $0<s<1$. Let $\Omega\subset\mathbb{R}^N$ be an open set. Then for every $u\in C^\infty_0(\Omega)$ we have
\begin{equation}
\label{poincare_inter}
\left(\lambda^1_{p}(\Omega)\right)^s\,\|u\|_{L^p(\Omega)}^p\le s\,(1-s)\,\|u\|_{\mathcal{X}^{s,p}_0(\Omega)}^p.
\end{equation}
\end{lm}
\begin{proof}
We proceed in two stages: we first prove that 
\[
\|u\|_{L^p(\Omega)}^p\lesssim \int_0^{+\infty} \left(\frac{K(t,u,L^p(\Omega),L^p(\Omega))}{t^s}\right)^p\,\frac{dt}{t},
\]
and then we show that the last integral is estimated from above by the norm $\mathcal{X}^{s,p}_0(\Omega)$.
\vskip.2cm\noindent
{\bf First stage.} Let us take $u\in C^\infty_0(\Omega)$, for every $t\ge 1$ and $v\in C^\infty_0(\Omega)$
\[
\|u\|_{L^p(\Omega)}\le \|u-v\|_{L^p(\Omega)}+t\,\|v\|_{L^p(\Omega)}.
\] 
By taking the infimum, we thus get
\[
\begin{split}
\|u\|_{L^p(\Omega)}\le  K(t,u,L^p(\Omega),L^p(\Omega)).
\end{split}
\]
By integrating with respect to the singular measure $dt/t$, we then get
\begin{equation}
\label{s0}
\int_1^{+\infty} \left(\frac{K(t,u,L^p(\Omega),L^p(\Omega))}{t^s}\right)^p\,\frac{dt}{t}\ge \int_1^{+\infty} t^{-s\,p}\,\|u\|^p_{L^p(\Omega)}\,\frac{dt}{t}=\frac{\|u\|_{L^p(\Omega)}^p}{s\,p}.
\end{equation}
We now pick $0<t<1$, by triangle inequality we get for every $v\in C^\infty_0(\Omega)$
\[
\begin{split}
t\,\|u\|_{L^p(\Omega)}&\le t\,\|u-v\|_{L^p(\Omega)}+t\,\|v\|_{L^p(\Omega)}\\
&\le \|u-v\|_{L^p(\Omega)}+t\,\|v\|_{L^p(\Omega)}.
\end{split}
\]
By taking the infimum over $v\in C^\infty_0(\Omega)$, we obtain for $u\in C^\infty_0(\Omega)$ and $0<t<1$
\[
t\,\|u\|_{L^p(\Omega)}\le K(t,u,L^p(\Omega),L^p(\Omega)).
\]
By integrating again, we get this time
\begin{equation}
\label{s1}
\int_0^{1} \left(\frac{K(t,u,L^p(\Omega),L^p(\Omega))}{t^s}\right)^p\,\frac{dt}{t}\ge \int_0^1 t^{p-s\,p}\,\|u\|^p_{L^p(\Omega)}\,\frac{dt}{t}=\frac{\|u\|_{L^p(\Omega)}^p}{(1-s)\,p}.
\end{equation}
By summing up \eqref{s0} and \eqref{s1}, we get the estimate
\begin{equation}
\label{1ststage}
\|u\|_{L^p(\Omega)}^p\le s\,(1-s)\,\int_0^{+\infty} \left(\frac{K(t,u,L^p(\Omega),L^p(\Omega))}{t^s}\right)^p\,\frac{dt}{t}.
\end{equation}
\vskip.2cm\noindent
{\bf Second stage.} Given $u\in C^\infty_0(\Omega)$, we take $v\in C^\infty_0(\Omega)$. We can suppose that $\lambda^1_p(\Omega)>0$, otherwise \eqref{poincare_inter} trivially holds.
By definition of $\lambda^1_{p}(\Omega)$ we have that
$$
\|u-v\|_{L^p(\Omega)}+t\,\|v\|_{L^p(\Omega)} \leq \|u-v\|_{L^p(\Omega)}+t\, (\lambda_{p}^1(\Omega))^{-\frac{1}{p}}\,\|\nabla v\|_{L^p(\Omega)}.
$$
If we recall the definition \eqref{K} of the $K-$functional, we get
$$
K(t,u,L^p(\Omega),L^p(\Omega))^p \leq \left( \|u-v\|_{L^p(\Omega)}+\frac{t}{ (\lambda_{p}^1(\Omega))^\frac{1}{p}}\|\nabla v\|_{L^p(\Omega)} \right)^p,
$$	
and by taking infimum over $v\in C^\infty_0(\Omega)$ and multiplying by $t^{-s\,p}$, we get
\[
t^{-s\,p}K(t,u,L^p(\Omega),L^p(\Omega))^p \leq t^{-s\,p}\, K\left(\frac{t}{ (\lambda_{p}^1(\Omega))^\frac{1}{p}},u\,,L^p(\Omega),\mathcal{D}^{1,p}_0(\Omega)\right)^p.
\]
We integrate over $t>0$, by performing the change of variable $\tau=t/(\lambda_{p}^1(\Omega))^\frac1p$ we get
\[
\begin{split}
\int_0^{+\infty}& \left(\frac{K(t,u,L^p(\Omega),L^p(\Omega))}{t^s}\right)^p\,\frac{dt}{t}\le \frac{1}{(\lambda_{p}^1(\Omega))^s } \|u\|^p_{\mathcal{X}^{s,p}_0(\Omega)}.
\end{split}
\]
By using this in \eqref{1ststage}, we prove the desired inequality \eqref{poincare_inter}.
\end{proof}
We will set
\[
\Lambda^s_p(\Omega)=\inf_{u\in C^\infty_0(\Omega)}\Big\{\|u\|^p_{\mathcal{X}^{s,p}_0(\Omega)}\, :\, \|u\|_{L^p(\Omega)}=1\Big\},
\]
i.e. this is the sharp constant in the relevant Poincar\'e inequality.
As a consequence of \eqref{poincare_inter}, we obtain
\begin{equation}
\label{mah0}
\Big(\lambda^1_{p}(\Omega)\Big)^s\le s\,(1-s)\,\Lambda^s_p(\Omega).
\end{equation}
\begin{prop}[Interpolation inequality]
\label{prop:2bis}
Let $1<p<\infty$ and $0<s<1$. Let $\Omega\subset\mathbb{R}^N$ be an open set. For every $u\in C^\infty_0(\Omega)$ we have
\begin{equation}
\label{interpol}
s\,(1-s)\,\|u\|_{\mathcal{X}^{s,p}_0(\Omega)}^p\le \|u\|_{L^p(\Omega)}^{p\,(1-s)}\,\|\nabla u\|_{L^p(\Omega)}^{s\,p}.
\end{equation}
In particular, we also obtain
\begin{equation}
\label{mah!!}
s\,(1-s)\,\Lambda^s_{p}(\Omega)\le \Big(\lambda^1_{p}(\Omega)\Big)^s.
\end{equation}
\end{prop}
\begin{proof}
We can assume that $u\not \equiv 0$, otherwise there is nothing prove.
In the definition of the $K-$functional $K(t,u,L^p(\Omega),\mathcal{D}^{1,p}_0(\Omega))$ we take $v=\tau\,u$ for $\tau >0$, thus we obtain
\[
\begin{split}
K(t,u,L^p(\Omega),\mathcal{D}^{1,p}_0(\Omega))&\le \inf_{\tau>0} \Big[|1-\tau|\,\|u\|_{L^p(\Omega)}+t\,\tau\,\|\nabla u\|_{L^p(\Omega)}\Big]\\
&=\min\Big\{\|u\|_{L^p(\Omega)},\,t\,\|\nabla u\|_{L^p(\Omega)}\Big\} .
\end{split}
\]
By integrating for $t>0$, we get
\[
\begin{split}
\|u\|_{\mathcal{X}^{s,p}_0(\Omega)}^p&\le \int_0^{+\infty} \frac{\min\Big\{\|u\|^p_{L^p(\Omega)},\,t^p\,\|\nabla u\|^p_{L^p(\Omega)}\Big\}}{t^{s\,p}}\,\frac{dt}{t}\\
&=\|\nabla u\|_{L^p(\Omega)}^p\,\int_0^\frac{\|u\|_{L^p(\Omega)}}{\|\nabla u\|_{L^p(\Omega)}} t^{p\,(1-s)}\,\frac{dt}{t}\\
&\quad +\|u\|_{L^p(\Omega)}^p\,\int_\frac{\|u\|_{L^p(\Omega)}}{\|\nabla u\|_{L^p(\Omega)}}^{+\infty} t^{-s\,p}\,\frac{dt}{t}\\
&=\|u\|_{L^p(\Omega)}^{p\,(1-s)}\,\|\nabla u\|_{L^p(\Omega)}^{s\,p}\,\left[\frac{1}{p\,(1-s)}+\frac{1}{s\,p}\right].
\end{split}
\]
We thus get the desired conclusion \eqref{interpol}. The estimate \eqref{mah!!} easily follows from the definition of Poincar\'e constant.
\end{proof}
From \eqref{mah0} and \eqref{mah!!}, we get in particular the following 
\begin{coro}[Equivalence of Poincar\'e constants]
\label{coro:uguali}
Let $1<p<\infty$ and $0<s<1$. For every $\Omega\subset\mathbb{R}^N$ open set we have
\[
s\,(1-s)\,\Lambda^s_p(\Omega)=\Big(\lambda^1_p(\Omega)\Big)^s.
\]
In particular, there holds
\[
\mathcal{D}^{1,p}_0(\Omega) \hookrightarrow L^p(\Omega) \qquad \Longleftrightarrow \qquad \mathcal{X}^{s,p}_0(\Omega) \hookrightarrow L^p(\Omega).
\]
\end{coro}
\begin{rem}[Extensions by zero in $\mathcal{X}^{s,p}_0$]
\label{rem:zeroext}
We observe that by interpolating the ``extension by zero'' operators
\[
\mathcal{T}_0 : \mathcal{D}^{1,p}_0(\Omega) \to \mathcal{D}^{1,p}_0(\mathbb{R}^N)\\
\qquad \mbox{ and }\qquad  \mathcal{T}_0:L^p(\Omega) \to L^{p}(\mathbb{R}^N)
\]
which are both continuous, one obtains the same result for the interpolating spaces. In other words, we have 
\[
\|u\|^p_{\mathcal{X}^{s,p}_0(\mathbb{R}^N)}\le \|u\|^p_{\mathcal{X}^{s,p}_0(\Omega)},\qquad \mbox{ for every }u\in C^\infty_0(\Omega).
\]
This can be also seen directly: it is sufficient to observe that $C^\infty_0(\Omega)\subset C^\infty_0(\mathbb{R}^N)$, thus we immediately get
\[
K(t,u,L^p(\mathbb{R}^N),\mathcal{D}^{1,p}_0(\mathbb{R}^N))\le K(t,u,L^p(\Omega),\mathcal{D}^{1,p}_0(\Omega)),
\]
since in the $K-$functional on the left-hand side the infimum is performed on a larger class.
By integrating, we get the conclusion.
\par
However, differently from the case of $\mathcal{D}^{1,p}_0(\Omega)$, $L^p(\Omega)$ and $\mathcal{D}^{s,p}_0(\Omega)$, in general for $u\in C^\infty_0(\Omega)$ we have
\[
\|u\|^p_{\mathcal{X}^{s,p}_0(\mathbb{R}^N)}< \|u\|^p_{\mathcal{X}^{s,p}_0(\Omega)}. 
\]
In other words, even if $u\equiv 0$ outside $\Omega$, passing from $\Omega$ to $\mathbb{R}^N$ has an impact on the interpolation norm.
\par
Actually, if $\Omega$ has not smooth boundary, the situation can be much worse than this. We refer to Remark \ref{rem:nightmare} below.
\end{rem}

\section{Interpolation VS. Sobolev-Slobodecki\u{\i}}
\label{sec:concrete}
\subsection{General sets}
We want to compare the norms of $\mathcal{D}^{s,p}_0(\Omega)$ and $\mathcal{X}^{s,p}_0(\Omega)$. We start with the simplest estimate, which is valid for every open set.
\begin{prop}[Comparison of norms I]
\label{lm:3}
Let $1<p<\infty$ and $0<s<1$. Let $\Omega\subset\mathbb{R}^N$ be an open set, then for every $u\in C^\infty_0(\Omega)$ we have
\begin{equation}
\label{mah}
\frac{1}{2^{p\,(1-s)}\,N\,\omega_N}\,\|u\|^p_{\mathcal{D}^{s,p}_0(\Omega)}\le \|u\|^p_{\mathcal{X}^{s,p}_0(\Omega)}.
\end{equation}
In particular, we have the continuous inclusion $\mathcal{X}^{s,p}_0(\Omega)\subset\mathcal{D}^{s,p}_0(\Omega)$.
\end{prop}
\begin{proof}
To prove \eqref{mah}, we take $h\in\mathbb{R}^N\setminus\{0\}$ and $\varepsilon>0$, then there exists $v\in C^\infty_0(\Omega)$ such that
\begin{equation}
\label{dai}
\|u-v\|_{L^p(\Omega)}+|h|\,\|v\|_{\mathcal{D}^{1,p}_0(\Omega)}\le (1+\varepsilon)\,K(|h|,u,L^p(\Omega),\mathcal{D}^{1,p}_0(\Omega)).
\end{equation}
Thus for $h\not=0$ we get\footnote{In the second inequality, we use the classical fact
\[
\begin{split}
\int_{\mathbb{R}^N} |\varphi(x+h)-\varphi(x)|^{p}\,dx&=\int_{\mathbb{R}^N} \left|\int_0^1\langle \nabla \varphi(x+t\,h),h\rangle\,dt\right|^{p}\,dx\\
&\le |h|^p\,\int_{\mathbb{R}^N}\int_0^1 |\nabla\varphi(x+t\,h)|^p\,dt\,dx\\
&= |h|^p\,\int_0^1 \left(\int_{\mathbb{R}^N}|\nabla\varphi(x+t\,h)|^p\,dx\right)\,dt=|h|^p\,\|\nabla \varphi\|_{L^p(\mathbb{R}^N)}.
\end{split}
\]}
\[
\begin{split}
\left(\int_{\mathbb{R}^N} \frac{|u(x+h)-u(x)|^{p}}{|h|^{N+s\,p}}\,dx\right)^\frac{1}{p}&\le \left(\int_{\mathbb{R}^N} \frac{|u(x+h)-v(x+h)-u(x)+v(x)|^{p}}{|h|^{N+s\,p}}\,dx\right)^\frac{1}{p}\\
&+\left(\int_{\mathbb{R}^N} \frac{|v(x+h)-v(x)|^{p}}{|h|^{N+s\,p}}\,dx\right)^\frac{1}{p}\\
&\le 2\,|h|^{-\frac{N}{p}-s}\,\|u-v\|_{L^p(\Omega)}\\
&+|h|^{1-\frac{N}{p}+s}\,\|\nabla v\|_{L^p(\Omega)}\\
&\le 2\,|h|^{-\frac{N}{p}-s}\,\left(\|u-v\|_{L^{p}(\Omega)}+\frac{|h|}{2}\,\|v\|_{\mathcal{D}^{1,p}_0(\Omega)}\right).
\end{split}
\]
By using \eqref{dai}, we then obtain
\[
\int_{\mathbb{R}^N} \frac{|u(x+h)-u(x)|^{p}}{|h|^{N+s\,p}}\,dx\le 2^p\,(1+\varepsilon)^p\, \left(\frac{K(|h|/2,u,L^p(\Omega),\mathcal{D}^{1,p}_0(\Omega))}{|h|^s}\right)^p\,\frac{1}{|h|^N}.
\]
We now integrate with respect to $h\in\mathbb{R}^N$. This yields
\[
\begin{split}
\iint_{\mathbb{R}^N\times\mathbb{R}^N} \frac{|u(x+h)-u(x)|^{p}}{|h|^{N+s\,p}}\,dx\,dh&\le 2^p\,(1+\varepsilon)^p\, \int_{\mathbb{R}^N} \left(\frac{K(|h|/2,u,L^p(\Omega),\mathcal{D}^{1,p}_0(\Omega))}{|h|^s}\right)^{p}\,\frac{dh}{|h|^N}\\
&=2^p\,(1+\varepsilon)^p\,N\,\omega_N\, \int_0^{+\infty} \left(\frac{K(t/2,u,L^p(\Omega),\mathcal{D}^{1,p}_0(\Omega))}{t^s}\right)^{p}\,\frac{dt}{t}.
\end{split}
\]
By making the change of variable $t/2=\tau$ and exploiting the arbitariness of $\varepsilon>0$, we eventually reach the desired estimate.
\end{proof}
\begin{coro}[Interpolation inequality for $\mathcal{D}^{s,p}_0$]
Let $1<p<\infty$ and $0<s<1$. Let $\Omega\subset\mathbb{R}^N$ be an open set. For every $u\in C^\infty_0(\Omega)$ we have
\begin{equation}
\label{mah!}
s\,(1-s)\,\|u\|_{\mathcal{D}^{s,p}_0(\Omega)}^p\le 2^{p\,(1-s)}\,N\,\omega_N\,\|u\|_{L^p(\Omega)}^{p\,(1-s)}\,\|\nabla u\|_{L^p(\Omega)}^{s\,p}.
\end{equation}
\end{coro}
\begin{proof}
It is sufficient to combine Propositions \ref{lm:3} and \ref{prop:2bis}.
\end{proof}
\begin{rem}
For $p\searrow 1$, the previous inequality becomes \cite[Proposition 4.2]{BLP}. In this case, the constant in \eqref{mah!} is sharp for $N=1$.
\end{rem}
For a general open set $\Omega\subset\mathbb{R}^N$, the converse of inequality \eqref{mah} does not hold.
This means that in general we have 
\[
\mathcal{X}^{s,p}_0(\Omega)\subset \mathcal{D}^{s,p}_0(\Omega) \qquad \mbox{ and }\qquad 
\mathcal{X}^{s,p}_0(\Omega)\not= \mathcal{D}^{s,p}_0(\Omega),
\]
the inclusion being continuous.
We use the construction of Appendix \ref{sec:example} in order to give a counter-example.
\begin{exa}
\label{exa:counternorma}
With the notation of Appendix \ref{sec:example}, let us take\footnote{In dimension $N=1$, we simply take $E=\mathbb{R}\setminus\mathbb{Z}$.}
\[
E=\mathbb{R}^N\setminus \left(\bigcup_{z\in\mathbb{Z}^N} (F+z)\right),\qquad \mbox{ with }F=\left[-\frac{1}{4},\frac{1}{4}\right]^{N-1}\times\{0\}.
\]
For every $\varepsilon>0$, we take $u_n\in C^\infty_0(\widetilde\Omega_n)\subset C^\infty_0(E)$ such that
\[
[u_n]^p_{W^{s,p}(\mathbb{R}^N)}<\lambda^s_{p}(\widetilde\Omega_n)+\varepsilon\qquad \mbox{ and }\qquad \int_{E} |u_n|^p\,dx=1.
\]
Here the set $\widetilde\Omega_n$ is defined by
\[
\widetilde\Omega_n=\bigcup_{z\in\mathbb{Z}^N_n} \Big(\Omega+z\Big)=\left[-n-\frac{1}{2},n+\frac{1}{2}\right]^N\setminus \bigcup_{z\in\mathbb{Z}^N_n} (F+z).
\]
On the other hand, we have
\[
\begin{split}
\|u_n\|^p_{\mathcal{X}^{s,p}_0(E)}&\ge \frac{\Big(\lambda_{p}^1(E)\Big)^s}{s\,(1-s)}\,\int_{E} |u_n|^p\,dx\ge\frac{\Big(\mu_p(Q;F)\Big)^s}{s\,(1-s)},
\end{split}
\]
where we also used \eqref{E}. 
By Lemma \ref{lm:bas}, we have that $\lambda^s_{p}(\Omega_n)$ converges to $0$ for $s\,p<1$, so that
\[
\liminf_{n\to\infty}\|u_n\|_{\mathcal{X}^{s,p}_0(E)}^p\ge \frac{1}{C}\qquad \mbox{ and }\qquad \limsup_{n\to\infty}\, [u_n]^p_{W^{s,p}(\mathbb{R}^N)}\le \varepsilon.
\] 
Thus by the arbitrariness of $\varepsilon$, we obtain
\[
\lim_{n\to\infty} \frac{\displaystyle\|u_n\|^p_{\mathcal{D}^{s,p}_0(E)}}{\|u_n\|^p_{\mathcal{X}^{s,p}_0(E)}}=0,\qquad \mbox{ for } 1<p<\infty \mbox{ and } s< \frac{1}{p}.
\]
\end{exa}
\begin{rem}[Extensions by zero in $\mathcal{X}^{s,p}_0(\Omega)$...reprise]
\label{rem:nightmare}
We take the set $E\subset\mathbb{R}^N$ and the sequence $\{u_n\}_{n\in\mathbb{N}}\subset C^\infty_0(E)$ as in Example \ref{exa:counternorma}. We have seen that
\[
\lim_{n\to\infty} \frac{\|u_n\|_{\mathcal{X}^{s,p}_0(E)}}{\|u_n\|_{\mathcal{D}^{s,p}_0(E)}}=+\infty.
\]
By using Proposition \ref{lm:Rn} we obtain
\[
\lim_{n\to\infty} \frac{\|u_n\|_{\mathcal{X}^{s,p}_0(E)}}{\|u_n\|_{\mathcal{X}^{s,p}_0(\mathbb{R}^N)}}=+\infty,
\]
as well, still for $s\,p<1$. Thus the ``extension by zero'' operator in general is not continuous.
\end{rem}
\subsection{Convex sets}
We now prove the converse of \eqref{mah}, under suitable assumptions on $\Omega$. We start with the case of a {\it convex} set. The case $\Omega=\mathbb{R}^N$ is simpler and instructive, thus we give a separate statement. The proof can be found for example in \cite[Lemma 35.2]{T}. We reproduce it, for the reader's convenience. We also single out an explicit determination of the constant.
\begin{prop}[Comparison of norms II: $\mathbb{R}^N$]
\label{lm:Rn}
Let $1<p<\infty$ and $0<s<1$. For every $u\in C^\infty_0(\mathbb{R}^N)$ we have
\[
\|u\|^p_{\mathcal{X}^{s,p}_0(\mathbb{R}^N)}\le \Big(N\,(N+1)\Big)^p\,\frac{2^p}{N\,\omega_N}\,\|u\|^p_{\mathcal{D}_0^{s,p}(\mathbb{R}^N)}.
\]
In particular, we have that $\mathcal{D}^{s,p}_0(\mathbb{R}^N)=\mathcal{X}^{s,p}_0(\mathbb{R}^N)$.
\end{prop}
\begin{proof}
Let $u\in C^\infty_0(\mathbb{R}^N)$, we set
\[
U(h)=\left(\int_{\mathbb{R}^N} |u(x+h)-u(x)|^p\,dx\right)^\frac{1}{p},\qquad h\in\mathbb{R}^N,
\]
and observe that by construction
\[
\int_{\mathbb{R}^N} \frac{U(h)^p}{|h|^{N+s\,p}}\,dh=[u]_{W^{s,p}(\mathbb{R}^N)}^p.
\]
We also define
\[
\overline{U}(\varrho)=\frac{1}{N\,\omega_N\,\varrho^{N-1}}\,\int_{\{h\in\mathbb{R}^N\, :\, |h|=\varrho\}} U\,d\mathcal{H}^{N-1},\qquad \varrho>0,
\]
thus by Jensen's inequality we have
\begin{equation}
\label{opla}
\begin{split}
\int_0^{+\infty} \left(\frac{\overline U}{\varrho^s}\right)^p\,\frac{d\varrho}{\varrho}&\le \frac{1}{N\,\omega_N}\,\int_{0}^{+\infty} \left(\int_{\{h\in\mathbb{R}^N\, :\, |h|=\varrho\}} U^p\,d\mathcal{H}^{N-1}\right)\,\frac{d\varrho}{\varrho^{N+s\,p}}\\
&= \frac{1}{N\,\omega_N}\int_{\mathbb{R}^N} \frac{U(h)^p}{|h|^{N+s\,p}}\,dh=\frac{1}{N\,\omega_N}\,[u]_{W^{s,p}(\mathbb{R}^N)}^p.
\end{split}
\end{equation}
We now take the compactly supported Lipschitz function 
\[
\psi(x)=\frac{N+1}{\omega_N}\,(1-|x|)_+,
\] 
where $(\,\cdot\,)_+$ stands for the positive part.
Observe that $\psi$ has unit $L^1$ norm, by construction. We then define
\[
\psi_t(x)=\frac{1}{t^N} \,\psi\left(\frac{x}{t}\right),\qquad \mbox{ for }t>0.
\]
From the definition of the $K-$functional, we get
\[
K(t,u,L^p(\mathbb{R}^N),\mathcal{D}^{1,p}_0(\mathbb{R}^N))\le \|u-\psi_t\ast u\|_{L^p(\Omega)}+t\,\|\nabla \psi_t\ast u\|_{L^p(\Omega)},
\]
by observing that $\psi_t\ast u\in C^\infty_0(\mathbb{R}^N)$.
We estimate the two norms separately: for the first one, by Minkowski inequality we get
\[
\begin{split}
\|u-\psi_t\ast u\|_{L^p(\mathbb{R}^N)}&=\left\|\int_{\mathbb{R}^N} [u(\cdot)-u(\cdot-y)]\,\psi_t(y)\,dy\right\|_{L^p(\mathbb{R}^N)}\\
&\le \int_{\mathbb{R}^N}\left(\int_{\mathbb{R}^N} |u(x)-u(x-y)|^p\,dx\right)^\frac{1}{p}\,\psi_t(y)\,dy\\
&=\int_{\mathbb{R}^N} U(-y)\,\psi_{t}(y)\,dy\le \frac{N+1}{\omega_N\,t^N}\,\int_{B_t(0)} U(-y)\,dy\\
&=\frac{N\,(N+1)}{t^N}\,\int_0^t \overline{U}\,\varrho^{N-1}\,d\varrho\le\frac{N\,(N+1)}{t}\,\int_0^t \overline{U}\,d\varrho.
\end{split}
\]
For the norm of the gradient, we first observe that 
\[
\int_{\mathbb{R}^N} \nabla \psi_t(y)\,dy=0,
\]
thus we can write
\[
\nabla \psi_t\ast u=(\nabla \psi_t)\ast u=\int_{\mathbb{R}^N} \nabla \psi_t(y)\,[u(x-y)-u(x)]\,dy.
\]
Consequently, by Minkowski inequality we get
\[
\begin{split}
\|\nabla \psi_t\ast u\|_{L^p(\mathbb{R}^N)}&=\left\|\int_{\mathbb{R}^N} \nabla \psi_t(y)\,[u(\cdot-y)-u(\cdot)]\,dy\right\|_{L^p(\mathbb{R}^N)}\\
&\le \int_{\mathbb{R}^N}\left( \int_{\mathbb{R}^N} |u(x-y)-u(x)|^p\,dx\right)^\frac{1}{p}\,|\nabla \psi_t(y)|\,dy\\
&\le \frac{N+1}{\omega_N\,t^{N+1}} \,\int_{B_t(0)} U(-y)\,dy\le \frac{N\,(N+1)}{t^2}\,\int_0^t \overline{U}\,d\varrho.
\end{split}
\]
In conclusion, we obtained for every $t>0$
\begin{equation}
\label{stimaK}
K(t,u,L^p(\mathbb{R}^N),\mathcal{D}^{1,p}_0(\mathbb{R}^N))\le \frac{2\,N\,(N+1)}{t}\,\int_0^t \overline{U}\,d\varrho.
\end{equation}
If we integrate on $(0,T)$, the previous estimate gives
\[
\int_0^T \left(\frac{K(t,u,L^p(\mathbb{R}^N),\mathcal{D}^{1,p}_0(\mathbb{R}^N))}{t^s}\right)^p\,\frac{dt}{t}\le \Big(2\,N\,(N+1)\Big)^p\,\int_0^T \left(\int_0^t \overline{U}\,d\varrho\right)^p\,t^{-p-s\,p}\,\frac{dt}{t}.
\]
If we now use Lemma \ref{lm:hardy1D} with $\alpha=p+s\,p$ for the function
\[
t\mapsto \int_0^t \overline{U}\,d\varrho,
\]
we get
\[
\begin{split}
\int_0^T \left(\frac{K(t,u,L^p(\Omega),\mathcal{D}^{1,p}_0(\Omega))}{t^s}\right)^p\,\frac{dt}{t}&\le \left(\frac{2\,N\,(N+1)}{s+1}\right)^p\,\int_0^T \left(\frac{\overline{U}}{t^s}\right)^p\frac{dt}{t}\\
&\le  \left(\frac{2\,N\,(N+1)}{s+1}\right)^p\,\frac{1}{N\,\omega_N}\,[u]^p_{W^{s,p}(\mathbb{R}^N)},
\end{split}
\]
where we used \eqref{opla} in the second inequality.
By letting $T$ going to $+\infty$, we get the desired estimate.
\end{proof}
We denote by 
\[
R_\Omega=\sup_{x\in\Omega} \mathrm{dist}(x,\partial\Omega),
\]
the {\it inradius} of an open set $\Omega\subset\mathbb{R}^N$. This is the radius of the largest open ball inscribed in $\Omega$.
We introduce the {\it eccentricity} of an open bounded set $\Omega\subset\mathbb{R}^N$, defined by
\[
\mathcal{E}(\Omega)=\frac{\mathrm{diam\,}(\Omega)}{2\,R_\Omega}.
\]
By generalizing the construction used in \cite[Lemma A.6]{BS} for a ball, we have the following.
\begin{thm}[Comparison of norms II: bounded convex sets]
\label{thm:convex}
Let $1<p<\infty$ and $0<s<1$. If $\Omega\subset\mathbb{R}^N$ is an open bounded convex set, then for every $u\in C^\infty_0(\Omega)$ we have
\begin{equation}
\label{mah!conv}
\|u\|^p_{\mathcal{X}^{s,p}_0(\Omega)}\le C\,\|u\|^p_{\mathcal{D}_0^{s,p}(\Omega)},
\end{equation}
for a constant $C=C(N,p,\mathcal{E}(\Omega))>0$. In particular, we have $\mathcal{X}^{s,p}_0(\Omega)=\mathcal{D}^{s,p}_0(\Omega)$.
\end{thm}
\begin{proof}
The proof runs similarly to that of Proposition \ref{lm:Rn} for $\mathbb{R}^N$, but now we have to pay attention to boundary issues. Indeed, the function $\psi_t\ast u$ is not supported in $\Omega$, unless $t$ is sufficiently small, depending on $u$ itself. In order to avoid this, we need to perform a controlled scaling of the function.
By keeping the same notation as in the proof of Proposition \ref{lm:Rn}, we need the following modification: we take a point $x_0\in\Omega$ such that 
\[
\mathrm{dist}(x_0,\partial\Omega)=R_\Omega.
\]
Without loss of generality, we can assume that $x_0=0$.
Then we define the rescaled function
\[
u_t=u\left(\frac{R_\Omega}{R_\Omega-t}\,x\right),\qquad 0<t<\frac{R_\Omega}{2}.
\]
We observe that 
\[
\mathrm{support}(u_t)=\frac{R_\Omega-t}{R_\Omega}\,\Omega,
\]
and by Lemma \ref{lm:convex}, we have
\[
\mathrm{dist}\left(\frac{R_\Omega-t}{R_\Omega}\,\Omega,\partial\Omega\right)\ge \left(1-\frac{R_\Omega-t}{R_\Omega}\right)\,R_\Omega=t.
\]
This implies that
\[
\psi_t\ast u_t\in C^\infty_0(\Omega),\qquad \mbox{ for every } 0<t<\frac{R_\Omega}{2}.
\]
We can now estimate the $K-$functional by using the choice $v=\psi_t\ast u_t$, 
that is
\[
\begin{split}
K(t,u,L^p(\Omega),\mathcal{D}^{1,p}_0(\Omega))&\le \|u-\psi_t\ast u_t\|_{L^p(\Omega)}\\
&+t\,\|\nabla \psi_t\ast u_t\|_{L^p(\Omega)},\quad\mbox{ for every } 0<t<\frac{R_\Omega}{2}.
\end{split}
\]
Let us set
\[
\Omega_t=\{x\in\mathbb{R}^N\, :\, \mathrm{dist}(x,\Omega)<t\},
\]
then we have that for every $x\in\Omega$, 
\[
y\mapsto \psi_t(x-y)\quad \mbox{ has support contained in } \Omega_t.
\]
By using this and Jensen's inequality, we obtain
\[
\|u-\psi_t\ast u_t\|_{L^p(\Omega)}^p\le \int_{\Omega}\int_{\Omega_t} \left|u(x)-u\left(\frac{R}{R-t}\,y\right)\right|^p\,\frac{1}{t^N}\,\psi\left(\frac{x-y}{t}\right)\,dy\,dx.
\]
Thus by using a change of variable and Fubini Theorem we get 
\[
\begin{split}
\int_0^{R_\Omega/2} &\left(\frac{\|u-\psi_t\ast u_t\|_{L^p(\Omega)}}{t^s}\right)^{p}\,\frac{dt}{t}\\
&\le \int_0^{R_\Omega/2} \int_{\Omega}\int_{\Omega_t} t^{-s\,p}\,\left|u(x)-u\left(\frac{R_\Omega}{R_\Omega-t}\,y\right)\right|^{p}\,\frac{1}{t^N}\,\psi\left(\frac{x-y}{t}\right)\,dy\,dx\,\frac{dt}{t}\\
&=\left(\frac{R_\Omega-t}{R_\Omega}\right)^N\,\int_0^{R_\Omega/2} \int_{\Omega}\int_{\frac{R_\Omega}{R_\Omega-t}\,\Omega_t} t^{-s\,p}\,\left|u(x)-u(z)\right|^{p}\,\frac{1}{t^N}\,\psi\left(\frac{x}{t}-\frac{R_\Omega-t}{R_\Omega\,t}z\right)\,dz\,dx\,\frac{dt}{t}\\
&\le \int_{\Omega}\int_{\widetilde\Omega}\left|u(x)-u(z)\right|^{p}\left(\int_0^{R_\Omega/2} \,t^{-s\,p-N}\,\psi\left(\frac{x-z}{t}+\frac{z}{R_\Omega}\right)\,\frac{dt}{t}\right)\,dz\,dx,
\end{split}
\]
where we used that
\[
\frac{R_\Omega}{R_\Omega-t}\,\Omega_t\subset \widetilde\Omega:=2\,\Omega_{R_\Omega/2},\qquad \mbox{ for } 0<t<\frac{R_\Omega}{2}.
\]
We now observe that 
\[
\psi\left(\frac{x-z}{t}+\frac{z}{R_\Omega}\right)\not=0\qquad \Longleftrightarrow \qquad  \left|\frac{x-z}{t}+\frac{z}{R_\Omega}\right|<1,
\]
thus in particular
\[
\mbox{ if }\quad \left|\frac{x-z}{t}\right|\ge 1+\left|\frac{z}{R_\Omega}\right|\qquad \mbox{ then } \qquad \psi\left(\frac{x-z}{t}+\frac{z}{R_\Omega}\right)=0,
\]
i.e. for every $x\in \Omega$ and $z\in\widetilde\Omega$,
\[
\mbox{ if } \quad 0<t\le \frac{|x-z|}{1+\dfrac{|z|}{R_\Omega}}\qquad \mbox{ then } \qquad \psi\left(\frac{x-z}{t}+\frac{z}{R_\Omega}\right)=0.
\]
This implies that  for $x\in \Omega$ and $z\in \widetilde \Omega$ we get
\[
\begin{split}
\int_0^{R_\Omega/2} \,t^{-s\,p-N}\,\psi\left(\frac{x-z}{t}+\frac{z}{R_\Omega}\right)\,\frac{dt}{t}&\le \int_0^{+\infty} \,t^{-s\,p-N}\,\psi\left(\frac{x-z}{t}+\frac{z}{R_\Omega}\right)\,\frac{dt}{t}\\
&=\int_{\frac{|x-z|}{1+\frac{|z|}{R_\Omega}}}^{+\infty} \,t^{-s\,p-N}\,\psi\left(\frac{x-z}{t}+\frac{z}{R_\Omega}\right)\,\frac{dt}{t}\\
&\le \int_{\frac{|x-z|}{1+\frac{\mathrm{diam}(\widetilde\Omega)}{R_\Omega}}}^{+\infty} \,t^{-s\,p-N}\,\psi\left(\frac{x-z}{t}+\frac{z}{R_\Omega}\right)\,\frac{dt}{t}\\
&\le \frac{N+1}{\omega_N\,(N+s\,p)}\,\left(1+\frac{\mathrm{diam}(\widetilde\Omega)}{R_\Omega}\right)^{N+s\,p}\,|x-z|^{-N-s\,p}.
\end{split}
\]
Thus, we obtain
\begin{equation}
\label{secondo}
\begin{split}
\int_0^{R_\Omega/2} &\left(\frac{\|u-\psi_t\ast u_t\|_{L^{p}(\Omega)}}{t^s}\right)^{p}\,\frac{dt}{t}\\
&\le \frac{N+1}{\omega_N\,(N+s\,p)}\,\left(1+\frac{\mathrm{diam}(\widetilde\Omega)}{R_\Omega}\right)^{N+s\,p}\,\int_\Omega \int_{\widetilde\Omega} \frac{|u(x)-u(z)|^{p}}{|x-z|^{N+s\,p}}\,dx\,dz\\
&\le \frac{N+1}{\omega_N\,(N+s\,p)}\,\left(1+\frac{\mathrm{diam}(\widetilde\Omega)}{R_\Omega}\right)^{N+s\,p}\, \|u\|^{p}_{\mathcal{D}^{s,p}_0(\Omega)}.
\end{split}
\end{equation}
Observe that by construction
\[
\mathrm{diam}(\widetilde\Omega)=2\,\mathrm{diam}(\Omega_{R_\Omega/2})\le 2\,\Big(\mathrm{diam}(\Omega)+R_\Omega\Big).
\]
We now need to show that 
\begin{equation}
\label{terzo}
\int_0^{R_\Omega/2} t^{p}\,\left(\frac{\|\psi_t\ast u_t\|_{\mathcal{D}^{1,p}_0(\Omega)}}{t^s}\right)^p\,\frac{dt}{t}\le C\, \|u\|^{p}_{\mathcal{D}^{s,p}_0(\Omega)}.
\end{equation}
We first observe that 
\[
\begin{split}
\nabla \psi_t\ast u_t(x)&=\int_{\mathbb{R}^N} u\left(\frac{R_\Omega}{R_\Omega-t}\,y\right)\,\frac{1}{t^{N+1}}\,\nabla \psi\left(\frac{x-y}{t}\right)\,dy,
\end{split}
\]
and by the Divergence Theorem
\[
\int_{\mathbb{R}^N}\frac{1}{t^{N+1}}\,\nabla \psi\left(\frac{x-y}{t}\right)\,dy=0.
\]
Thus we obtain as well
\[
-\nabla \psi_t\ast u_t(x)=\int_{\mathbb{R}^N} \left[u\left(\frac{R_\Omega}{R_\Omega-t}\,x\right)-u\left(\frac{R_\Omega}{R_\Omega-t}\,y\right)\right]\,\frac{1}{t^{N+1}}\,\nabla \psi\left(\frac{x-y}{t}\right)\,dy,
\]
and by H\"older's inequality  
\[
\begin{split}
\|\psi_t\ast u_t\|_{\mathcal{D}^{1,p}_0(\Omega)}^{p}&=\int_{\mathbb{R}^N} |\nabla u_t|^{p}\,dx\\
& \le \int_{\mathbb{R}^N} \left(\int_{\mathbb{R}^N} \left|u\left(\frac{R_\Omega}{R_\Omega-t}\,x\right)-u\left(\frac{R_\Omega}{R_\Omega-t}\,y\right)\right|^{p}\,\frac{1}{t^{N+1}}\,\left|\nabla \psi\left(\frac{x-y}{t}\right)\right|\,dy\right)\\
&\times\left(\int_{\mathbb{R}^N}\frac{1}{t^{N+1}}\,\left|\nabla \psi\left(\frac{x-y}{t}\right)\right|\,dy\right)^{p-1}\,dx\\
&= \frac{\|\nabla \psi\|^{p-1}_{L^1(\mathbb{R}^N)}}{t^{p-1}}\,\int_{\mathbb{R}^N} \int_{\mathbb{R}^N} \left|u\left(\frac{R_\Omega}{R_\Omega-t}\,x\right)-u\left(\frac{R_\Omega}{R_\Omega-t}\,y\right)\right|^{p}\,\,\frac{1}{t^{N+1}}\,\left|\nabla \psi\left(\frac{x-y}{t}\right)\right|\,dy\,dx\\
&\le \frac{\|\nabla \psi\|^{p-1}_{L^1(\mathbb{R}^N)}}{t^{p-1}}\,\int_{\mathbb{R}^N} \int_{\mathbb{R}^N} \left|u\left(z\right)-u\left(w\right)\right|^{p}\,\,\frac{1}{t^{N+1}}\,\left|\nabla \psi\left(\frac{R_\Omega-t}{R_\Omega\,t}\,(z-w)\right)\right|\,dz\,dw.
\end{split}
\]
This yields 
\begin{equation}
\label{quasi6}
\begin{split}
\int_0^{R_\Omega/2} &t^{p}\,\left(\frac{\|u_t\|_{\mathcal{D}^{1,p}_0(\Omega)}}{t^s}\right)^{p}\,\frac{dt}{t}\\
&\le C\,\int_0^{R_\Omega/2} t^{-s\,p}\,\int_{\mathbb{R}^N} \int_{\mathbb{R}^N} \left|u(z)-u(w)\right|^{p}\,\,\frac{1}{t^{N}}\,\left|\nabla \psi\left(\frac{R_\Omega-t}{R_\Omega\,t}\,(z-w)\right)\right|\,dz\,dw\,\frac{dt}{t}\\
&=C\,\int_{\mathbb{R}^N} \int_{\mathbb{R}^N} |u(z)-u(w)|^{p}\,\left(\int_0^{R_\Omega/2} t^{-s\,p}\,\frac{1}{t^{N}}\,\left|\nabla \psi\left(\frac{R_\Omega-t}{R_\Omega\,t}\,(z-w)\right)\right|\,\frac{dt}{t}\right)\,dz\,dw.
\end{split}
\end{equation}
As above, we now observe that 
\[
\left|\nabla \psi\left(\frac{R_\Omega-t}{R_\Omega\,t}\,(z-w)\right)\right|\not=0\qquad \Longleftrightarrow\qquad \frac{R_\Omega-t}{R_\Omega}\,\frac{|z-w|}{t}<1,
\]
thus in particular for $0<t<R_\Omega/2$ we have
\[
\frac{1}{2}\,\frac{|z-w|}{t}>1\qquad \Longrightarrow\qquad \nabla \psi\left(\frac{R_\Omega-t}{R_\Omega\,t}\,(z-w)\right)=0.
\]
This implies that for $z,w\in\mathbb{R}^N$ we have
\[
\begin{split}
\int_0^{R_\Omega/2} t^{-s\,p}\,\frac{1}{t^{N}}\,\left|\nabla \psi\left(\frac{R_\Omega-t}{R_\Omega\,t}\,(z-w)\right)\right|\,\frac{dt}{t}&\le \int_{\frac{|z-w|}{2}}^{+\infty} t^{-s\,p}\,\frac{1}{t^{N}}\,\left|\nabla \psi\left(\frac{R_\Omega-t}{R_\Omega\,t}\,(z-w)\right)\right|\,\frac{dt}{t}\\
&\le \frac{N+1}{\omega_N\,(N+s\,p)}\,|z-w|^{-N-s\,p}.
\end{split}
\]
By inserting this estimate in \eqref{quasi6}, we now get \eqref{terzo}.
We are left with estimating the integral of the $K-$functional on $(R_\Omega/2,+\infty)$: for this, we can use the trivial decomposition 
\[
u=(u-0)+0,
\]
which gives
\[
\begin{split}
\int_{\frac{R_\Omega}{2}}^{+\infty} \left(\frac{K(t,u,L^p(\Omega),\mathcal{D}^{1,p}_0(\Omega))}{t^s}\right)^p\,\frac{dt}{t}&\le \int_{\frac{R_\Omega}{2}}^{+\infty} \frac{\|u\|_{L^p(\Omega)}^p}{t^{s\,p}}\,\frac{dt}{t}\\
&=\frac{\|u\|_{L^p(\Omega)}^p}{s\,p}\, \left(\frac{R_\Omega}{2}\right)^{-s\,p}\\
&\le \frac{2^{s\,p}}{s\,p}\,\|u\|_{\mathcal{D}^{s,p}_0(\Omega)}^p\, \left(\frac{1}{\lambda^s_p(\Omega)\,R_\Omega^{s\,p}}\right),
\end{split}
\]
where we used Poincar\'e inequality for $\mathcal{D}^{s,p}_0(\Omega)$. By recalling that for a convex set with finite inradius we have (see \cite[Corollary 5.1]{BC})
\[
\lambda^s_p(\Omega)\,R_\Omega^{s\,p}\ge \frac{\mathcal{C}}{s\,(1-s)},
\]
for a constant $\mathcal{C}=\mathcal{C}(N,p)>0$, we finally obtain
\[
\int_{\frac{R_\Omega}{2}}^{+\infty} \left(\frac{K(t,u,L^p(\Omega),\mathcal{D}^{1,p}_0(\Omega))}{t^s}\right)^p\,\frac{dt}{t}\le \frac{2^{s\,p}}{p}\,\|u\|_{\mathcal{D}^{s,p}_0(\Omega)}^p\, \left(\frac{1-s}{\mathcal{C}}\right).
\]
By using this in conjunction with \eqref{secondo} and \eqref{terzo}, we finally conclude the proof.
\end{proof}
For general unbounded convex sets, the previous proof does not work anymore. However, for convex {\it cones} the result still holds. We say that a convex set $\Omega\subset\mathbb{R}^N$ is a convex cone centered at $x_0\in\mathbb{R}^N$ if for every $x\in\Omega$ and $\tau>0$, we have
\[
x_0+\tau\,(x-x_0)\in \Omega.
\]
Then we have the following
\begin{coro}[Comparison of norms II: convex cones]
\label{coro:cones}
Let $1<p<\infty$ and $0<s<1$. If $\Omega\subset\mathbb{R}^N$ is an open convex cone centered at $x_0\in\mathbb{R}^N$,
then for every $u\in C^\infty_0(\Omega)$ we have
\[
\|u\|^p_{\mathcal{X}^{s,p}_0(\Omega)}\le C\,\|u\|^p_{\mathcal{D}_0^{s,p}(\Omega)},
\]
for a constant $C=C(N,p,\mathcal{E}(\Omega\cap B_1(x_0)))>0$. In particular, we have $\mathcal{X}^{s,p}_0(\Omega)=\mathcal{D}^{s,p}_0(\Omega)$.
\end{coro}
\begin{proof}
We assume for simplicity that $x_0=0$ and take $u\in C^\infty_0(\Omega)$. Since $u$ has compact support, we have that $u\in C^\infty_0(\Omega\cap B_R(0))$, for $R$ large enough. From the previous result, we know that
\[
\|u\|^p_{\mathcal{X}^{s,p}_0(\Omega\cap B_R(0))}\le C\,\|u\|_{\mathcal{D}^{s,p}_0(\Omega\cap B_R(0))}= C\, \|u\|_{\mathcal{D}^{s,p}_0(\Omega)}.
\]
We recall that the constant $C$ depends on the eccentricity of $\Omega\cap B_R(0)$. However, since $\Omega$ is a cone, we easily get
\[
\mathcal{E}(\Omega\cap B_R(0))=\mathcal{E}(\Omega\cap B_1(0)),\qquad \mbox{ for every } R>0,
\]
i.e. the constant $C$ is independent of $R$. Finally, by observing that 
\[
\|u\|^p_{\mathcal{X}^{s,p}_0(\Omega)}\le \|u\|^p_{\mathcal{X}^{s,p}_0(\Omega\cap B_R(0))},
\]
we get the desired conclusion.
\end{proof}
\begin{rem}[Rotationally symmetric cones]
Observe that if $\Omega$ is the rotationally symmetric convex cone
\[
\Omega=\{x\in\mathbb{R}^N\, :\, \langle x-x_0,\omega\rangle>\beta\,|x-x_0|\},\qquad \mbox{ for some }0\le \beta <1,\, x_0\in\mathbb{R}^N\mbox{ and }\omega\in\mathbb{S}^{N-1},
\]
we have
\[
\mathcal{E}(\Omega\cap B_1(0))=\frac{1}{2}\,\max\left\{2\,\sqrt{1-\beta^2},\, 1\right\}\,\left(1+\frac{1}{\sqrt{1-\beta^2}}\right),
\]
by elementary geometric considerations.
\par
In particular, when $\Omega$ is a half-space (i.e. $\beta=0$), then we have $\mathcal{E}(\Omega\cap B_1(0))=2$.
\end{rem}

\subsection{Lipschitz sets and beyond}
In this section we show that the norms of $\mathcal{X}^{s,p}_0$ and $\mathcal{D}^{s,p}_0$ are equivalent on open bounded Lipschitz sets. We also make some comments on more general sets, see Remark \ref{rem:vicious} below.
\vskip.2cm
By generalizing the idea of \cite[Theorem 11.6]{LM} (see also \cite[Theorem 2.1]{Bra}) for $p=2$ and smooth sets, we can rely on the powerful extension theorem for Sobolev functions proved by Stein and obtain the following
\begin{thm}[Comparison of norms II: Lipschitz sets]
\label{thm:finally}
Let $1<p<\infty$ and $0<s<1$. Let $\Omega\subset\mathbb{R}^N$ be an open bounded set, with Lipschitz boundary. Then for every $u\in C^\infty_0(\Omega)$ we have
\[
\|u\|^p_{\mathcal{X}^{s,p}_0(\Omega)}\le C_1\,\|u\|^p_{\mathcal{D}_0^{s,p}(\Omega)},
\]
for a constant $C_1>0$ depending on $N,p, \mathrm{diam}(\Omega)$ and the Lipschitz constant of $\partial\Omega$. In particular, we have $\mathcal{X}^{s,p}_0(\Omega)=\mathcal{D}^{s,p}_0(\Omega)$ in this case as well.
\end{thm}
\begin{proof}
We take an open ball $B\subset\mathbb{R}^N$ with radius $\mathrm{diam}(\Omega)$ and such that $\Omega\Subset B$.
We then take a linear and continuous extension operator
\[
\mathcal{T}: W^{1,p}(B\setminus\overline{\Omega})\to W^{1,p}(B),
\]
such that 
\begin{equation}
\label{stimajones}
\left\{\begin{array}{rcl}
\|\mathcal{T}(u)\|_{L^p(B)}&\le& \mathfrak{e}_\Omega\,\|u\|_{L^p(B)},\\
&&\\
\|\nabla \mathcal{T}(u)\|_{L^p(B)}&\le& \mathfrak{e}_\Omega\,\|u\|_{W^{1,p}(B)},
\end{array}
\right.
\end{equation}
where $\mathfrak{e}_\Omega>0$ depends on $N,p,\varepsilon,\delta$ and $\mathrm{diam}(\Omega)$.
We observe that such an operator exists, thanks to the fact that $\Omega$ has a Lipschitz boundary, see \cite[Theorem 5, page 181]{St}. We also observe that the first estimate in \eqref{stimajones} is not explicitly stated by Stein, but it can be extrapolated by having a closer look at the proof, see \cite[page 192]{St}.
\par
For every $v\in C^\infty_0(B)$, we define the operator 
\[
\mathcal{R}(v)=v-\mathcal{T}(v),
\]
and observe that 
\[
\mathcal{R}(v)\equiv 0 \mbox{ in }B\setminus \overline{\Omega}\qquad \mbox{ and }\qquad \mathcal{R}(v)\in W^{1,p}(B).
\]
Since $\Omega$ has continuous boundary, this implies that $\mathcal{R}(v)\in\mathcal{D}^{1,p}_0(\Omega)$, see Remark \ref{rem:density}. We now fix $u\in C^\infty_0(\mathbb{R}^N)$, for every $v\in C^\infty_0(B)$ and every $\varepsilon>0$, we take $\varphi_\varepsilon\in C^\infty_0(\Omega)$ such that
\[
\Big(\lambda^1_p(\Omega)\Big)^\frac{1}{p}\,\|\varphi_\varepsilon-\mathcal{R}(v)\|_{L^p(\Omega)}\le \|\nabla \varphi_\varepsilon-\nabla \mathcal{R}(v)\|_{L^p(\Omega)}<\varepsilon.
\]
This is possible, thanks to the definition of $\mathcal{D}^{1,p}_0(\Omega)$.
Then for $t>0$ we can estimate the relevant $K-$functional as follows
\[
\begin{split}
K(t,\mathcal{R}(u),L^p(\Omega),\mathcal{D}^{1,p}_0(\Omega))&\le \|\mathcal{R}(u)-\varphi_\varepsilon\|_{L^p(\Omega)}+t\,\|\nabla \varphi_\varepsilon\|_{L^p(\Omega)}\\
&\le \|\mathcal{R}(u)-\mathcal{R}(v)\|_{L^p(\Omega)}+\|\mathcal{R}(v)-\varphi_\varepsilon\|_{L^p(\Omega)}\\
&+t\,\|\nabla \mathcal{R}(v)\|_{L^p(\Omega)}+t\,\|\nabla \varphi_\varepsilon-\nabla \mathcal{R}(v)\|_{L^p(\Omega)}\\
&\le \|\mathcal{R}(u-v)\|_{L^p(\Omega)}+t\,\|\nabla \mathcal{R}(v)\|_{L^p(\Omega)}+\varepsilon\,\left(1+\left(\lambda^1_p(\Omega)\right)^{-\frac{1}{p}}\right)\\
&\le \|u-v\|_{L^p(\Omega)}+\|\mathcal{T}(u-v)\|_{L^p(\Omega)}\\
&+t\,\left(\|\nabla v\|_{L^p(\Omega)}+\|\nabla \mathcal{T}(v)\|_{L^p(\Omega)}\right)+\varepsilon\,\left(1+\left(\lambda^1_p(\Omega)\right)^{-\frac{1}{p}}\right).
\end{split}
\]
By applying \eqref{stimajones}, we then get
\[
\begin{split}
K(t,\mathcal{R}(u),L^p(\Omega),\mathcal{D}^{1,p}_0(\Omega))&\le (1+\mathfrak{e}_\Omega)\,\|u-v\|_{L^p(B)}+t\,\left(\|\nabla v\|_{L^p(B)}+\mathfrak{e}_\Omega\,\|v\|_{W^{1,p}(B)}\right)\\
&+\varepsilon\,\left(1+\left(\lambda^1_p(\Omega)\right)^{-\frac{1}{p}}\right).
\end{split}
\]
We now use that 
\[
\|v\|_{W^{1,p}(B)}=\left(\|v\|^p_{L^p(B)}+\|\nabla v\|^p_{L^p(B)}\right)^\frac{1}{p}\le \|\nabla v\|_{L^p(B)}\,\left(1+\frac{1}{\lambda^1_p(B)}\right)^\frac{1}{p},
\]
thanks to Poincar\'e inequality. By spending this information in the previous estimate and using the arbitrariness of $\varepsilon$, we get
\[
\begin{split}
K(t,\mathcal{R}(u),L^p(\Omega),\mathcal{D}^{1,p}_0(\Omega))&\le (1+\mathfrak{e}_\Omega)\,\|u-v\|_{L^p(B)}\\
&+t\,\left(1+\mathfrak{e}_\Omega\,\left(1+\frac{1}{\lambda^1_p(B)}\right)^\frac{1}{p}\right)\,\|\nabla v\|_{L^p(B)}.
\end{split}
\]
We set for simplicity
\[
\gamma_\Omega=1+\mathfrak{e}_\Omega\,\left(1+\frac{1}{\lambda^1_p(B)}\right)^\frac{1}{p},
\]
then by taking the infimum over $v\in C^\infty_0(B)$
\[
K(t,\mathcal{R}(u),L^p(\Omega),\mathcal{D}^{1,p}_0(\Omega))\le \gamma_\Omega\,K(t,u,L^p(\Omega),\mathcal{D}^{1,p}_0(\Omega)).
\] 
As usual, we integrate in $t$, so to get
\begin{equation}
\label{approx}
\|\mathcal{R}(u)\|^p_{\mathcal{X}^{s,p}_0(\Omega)}\le \gamma_\Omega^p\,\|u\|^p_{\mathcal{X}^{s,p}_0(B)},\qquad \mbox{ for } u\in C^\infty_0(\mathbb{R}^N).
\end{equation}
We now observe that if $u\in C^\infty_0(\Omega)$, then we have $\mathcal{R}(u)=u$. Thus from \eqref{approx} and Theorem \ref{thm:convex} for the convex set $B$, we get
\[
\|u\|^p_{\mathcal{X}^{s,p}_0(\Omega)}\le C\,\gamma_\Omega^p\,\|u\|^p_{\mathcal{D}^{s,p}_0(\mathbb{R}^N)}=C\,\gamma_\Omega^p\,\|u\|^p_{\mathcal{D}^{s,p}_0(\Omega)},\qquad \mbox{ for every }u\in C^\infty_0(\Omega),
\]
where $C$ only depends on $N$ and $p$.
This concludes the proof.
\end{proof}
\begin{rem}[More general sets]
\label{rem:vicious}
It is not difficult to see that the previous proof works (and thus $\mathcal{X}^{s,p}_0(\Omega)$ and $\mathcal{D}^{s,p}_0(\Omega)$ are equivalent), whenever the set $\Omega$ is such that there exists a linear and continuous extension operator 
\[
\mathcal{T}: W^{1,p}(B\setminus\overline{\Omega})\to W^{1,p}(B),
\]
such that \eqref{stimajones} holds. Observe that there is a vicious subtility here: the first condition in \eqref{stimajones} is vital and, in general, {\it it may fail to hold for an extension operator}. For example, there is a beautiful extension result by Jones \cite[Theorem 1]{Jo}, which is valid for very irregular domains (possibly having a fractal boundary): however, the construction given by Jones does not assure that the first estimate in \eqref{stimajones} holds true, see the statement of \cite[Lemma 3.2]{Jo}. 

\end{rem}
In order to complement the discussion of Remarks \ref{rem:zeroext} and \ref{rem:nightmare} on ``extensions by zero'' in $\mathcal{X}^{s,p}_0$, we explicitly state the consequence of \eqref{approx}.
\begin{coro}
Let $1<p<\infty$ and $0<s<1$. Let $\Omega\subset\mathbb{R}^N$ be an open bounded set as in Theorem \ref{thm:finally}. Then for every $u\in C^\infty_0(\Omega)$, there holds
\[
\|u\|^p_{\mathcal{X}^{s,p}_0(\Omega)}\le C\, \|u\|^p_{\mathcal{X}^{s,p}_0(\mathbb{R}^N)},
\]
for a constant $C=C(N,p,\varepsilon,\delta)>0$.
\end{coro}

\section{Capacities} 
\label{sec:capacities}
Let $1<p<\infty$ and $0<s<1$ be such that\footnote{As usual, the restriction $s\,p<N$ is due to the scaling properties of the relevant energies. It is not difficult to see that for $s\,p\ge N$, both infima are identically $0$.} $s\,p<N$. For every compact set $F\subset\mathbb{R}^N$, we define the {\it $(s,p)-$capacity of $F$ }
\[
\mathrm{cap}_{s,p}(F)=\inf_{u\in C^\infty_0(\mathbb{R}^N)}\left\{[u]^p_{W^{s,p}(\mathbb{R}^N)}\, :\, u\ge 0\mbox{ and } u\ge 1_F\right\},
\]
and the {\it interpolation $(s,p)-$capacity of $F$}
\[
\mathrm{int\,cap}_{s,p}(F)=\inf_{u\in C^\infty_0(\mathbb{R}^N)}\left\{\|u\|^p_{\mathcal{X}^{s,p}_0(\mathbb{R}^N)}\, :\, u\ge 0 \mbox{ and }u\ge 1_F\right\}.
\]
As a straightforward consequence of Propositions \ref{lm:3} and \ref{lm:Rn}, we have the following
\begin{coro}[Comparison of capacities]
\label{coro:zerocap}
Let $1<p<\infty$ and $0<s<1$ be such that $s\,p<N$. Let $F\subset\mathbb{R}^N$ be a compact set, then we have
\[
\frac{1}{C}\,\mathrm{cap}_{s,p}(F)\le \mathrm{int\,cap}_{s,p}(F)\le C\,\mathrm{cap}_{s,p}(F),
\]
for a constant $C=C(N,p)>1$.
In particular, it holds
\[
\mathrm{cap}_{s,p}(F)=0\qquad \mbox{ if and only if }\qquad \mathrm{int\,cap}_{s,p}(F)=0.
\]
\end{coro}

\begin{prop}
\label{prop:submodular}
Let $1<p<\infty$ and $0<s<1$ be such that $s\,p<N$. For every $E,F\subset\mathbb{R}^N$ compact sets, we have
\[
\mathrm{cap}_{s,p}(E\cup F)\le \mathrm{cap}_{s,p}(E)+\mathrm{cap}_{s,p}(F).
\]
\end{prop}
\begin{proof}
We fix $n\in\mathbb{N}\setminus\{0\}$ and choose two non-negative functions $\varphi_n,\psi_n\in C^\infty_0(\mathbb{R}^N)$ such that
\[
[\varphi_n]^p_{W^{s,p}(\mathbb{R}^N)}\le \mathrm{cap}_{s,p}(E)+\frac{1}{n},\qquad \varphi_n\ge 1_E,
\]
and
\[
[\psi_n]^p_{W^{s,p}(\mathbb{R}^N)}\le \mathrm{cap}_{s,p}(F)+\frac{1}{n},\qquad \psi_n\ge 1_F.
\]
We then set
\[
U_{n,\varepsilon}=\Big(\max\{\varphi_n,\psi_n\}\Big)\ast \varrho_\varepsilon,\qquad 0<\varepsilon\ll 1,
\]
where $\{\varrho_\varepsilon\}_{\varepsilon>0}$ is a family of standard Friedrichs mollifiers. We observe that for every $n\in\mathbb{N}\setminus\{0\}$, it holds that $U_{n,\varepsilon}\in C^\infty_0(\mathbb{R}^N)$. Moreover, by construction we have
\[
U_{n,\varepsilon}\ge 1_{E\cup F}.
\]
By observing that Jensen's inequality implies
\[
[U_{n,\varepsilon}]_{W^{s,p}(\mathbb{R}^N)}\le \Big[\max\{\varphi_n,\psi_n\}\Big]^p_{W^{s,p}(\mathbb{R}^N)},
\]
we thus get
\[
\mathrm{cap}_{s,p}(E\cup F)\le [U_{n,\varepsilon}]_{W^{s,p}(\mathbb{R}^N)}\le \Big[\max\{\varphi_n,\psi_n\}\Big]^p_{W^{s,p}(\mathbb{R}^N)}.
\]
By using the submodularity of the Sobolev-Slobodeck\u{\i}i seminorm (see \cite[Theorem 3.2 \& Remark 3.3]{GM}), we obtain
\[
\mathrm{cap}_{s,p}(E\cup F)\le [\varphi_n]^p_{W^{s,p}(\mathbb{R}^N)}+[\psi_n]^p_{W^{s,p}(\mathbb{R}^N)}.
\]
Finally, thanks to the choice of $\varphi_n$ and $\psi_n$, we get the desired conclusion by the arbitrariness of $n$.
\end{proof}

\begin{prop}
\label{prop:capdim}
Let $1<p<\infty$ and $0<s<1$ be such that $s\,p<N$. Let $\Omega\subset\mathbb{R}^N$ be an open set. We take a compact set $E\Subset\Omega$ such that
\[
\mathrm{cap}_{s,p}(E)=0.
\]
Then we have
\begin{equation}
\label{hausdorff}
\mathcal{H}^\tau(E)=0\qquad \mbox{ for every }\tau>N-s\,p,
\end{equation}
and
\begin{equation}
\label{cap_poincare}
\lambda^s_p(\Omega\setminus E)=\lambda^s_p(\Omega).
\end{equation}
\end{prop}
\begin{proof}
To prove \eqref{hausdorff}, we can easily adapt the proof of \cite[Theorem 4, page 156]{EG}, dealing with the local case. 
\vskip.2cm\noindent
In order to prove \eqref{cap_poincare}, we first assume $\Omega$ to be bounded.
Let $\varepsilon>0$, we take $u_\varepsilon\in C^\infty_0(\Omega)$ such that
\[
\|u_\varepsilon\|_{\mathcal{D}^{s,p}_0(\Omega)}^p<(1+\varepsilon)\,\lambda^s_p(\Omega)\qquad \mbox{ and }\qquad \int_\Omega |u_\varepsilon|^p\,dx=1.
\]
We further observe that the boundedness of $\Omega$ implies that 
\[
\lambda^s_p(\Omega)=\min_{u\in\mathcal{D}^{s,p}_0(\Omega)} \Big\{\|u\|^p_{\mathcal{D}^{s,p}_0(\Omega)}\, :\, \|u\|_{L^p(\Omega)}=1\Big\},
\]
and that any solution of this problem has norm $L^\infty(\Omega)$ bounded by a universal constant, see \cite[Theorem 3.3]{BLP}.
Thus, without loss of generality, we can also assume that 
\[
\|u_\varepsilon\|_{L^\infty(\Omega)}\le M,\qquad \mbox{ for } 0<\varepsilon\ll1.
\]
Since $E$ has null $(s,p)-$capacity, there exists $\varphi_\varepsilon\in C^\infty_0(\Omega)$ such that
\[
[\varphi_\varepsilon]_{W^{s,p}(\mathbb{R}^N)}^p<\varepsilon,\qquad \varphi_\varepsilon\ge 0\qquad \mbox{ and }\qquad \varphi_\varepsilon\ge 1_E.
\]
We set $\psi_\varepsilon=\varphi_\varepsilon/\|\varphi_\varepsilon\|_{L^\infty(\mathbb{R}^N)}$ and observe that $\|\varphi_\varepsilon\|_{L^\infty(\mathbb{R}^N)}\ge 1$.
The function $u_\varepsilon\,(1-\psi_\varepsilon)$ is admissible for the variational problem defining $\lambda^s_p(\Omega\setminus E)$, then by using the triangle inequality we have
\[
\begin{split}
\Big(\lambda^s_{p}(\Omega\setminus E)\Big)^\frac{1}{p}\le \frac{[u_\varepsilon\,(1-\psi_\varepsilon)]_{W^{s,p}(\mathbb{R}^N)}}{\|u_\varepsilon\,(1-\psi_\varepsilon)\|_{L^p(\Omega\setminus E)}}&\le \frac{[u_\varepsilon]_{W^{s,p}(\mathbb{R}^N)}\,\|1-\psi_\varepsilon\|_{L^\infty(\mathbb{R}^N)}}{\|u_\varepsilon\,(1-\psi_\varepsilon)\|_{L^p(\Omega\setminus E)}}\\
&+\frac{\|u_\varepsilon\|_{L^\infty}\,[\psi_\varepsilon]_{W^{s,p}(\mathbb{R}^N)}}{\|u_\varepsilon\,(1-\psi_\varepsilon)\|_{L^p(\Omega\setminus E)}}.
\end{split}
\]
From the first part of the proof, we know that $E$ has Lebesgue measure $0$, thus the $L^p$ norm over $\Omega\setminus E$ is the same as that over $\Omega$.
If we now take the limit as $\varepsilon$ goes to $0$ and use the properties of $u_\varepsilon$, together with\footnote{Observe that, from the first condition, we get that $\psi_\varepsilon$ converges to $0$ strongly in $L^p(\Omega)$, by Sobolev inequality. Since the family $\{u_\varepsilon\}$ is bounded in $L^\infty(\Omega)$, this is enough to infer
\[
\lim_{\varepsilon\to 0} \int_{\Omega} |u_\varepsilon|^p\,|1-\psi_\varepsilon|^p\,dx=1.
\]}
\[
[\psi_\varepsilon]_{W^{s,p}(\mathbb{R}^N)}<\frac{\varepsilon}{\|\varphi_\varepsilon\|_{L^\infty(\mathbb{R}^N)}}\le \varepsilon,
\]
and
\[
\lim_{\varepsilon \to 0}\|1-\psi_\varepsilon\|_{L^\infty(\mathbb{R}^N)}=\lim_{\varepsilon\to 0} \sup_{\mathbb{R}^N}\, (1-\psi_\varepsilon)\le 1,
\] 
we get
\[
\Big(\lambda^s_{p}(\Omega\setminus E)\Big)^\frac{1}{p}\le \Big(\lambda^s_{p}(\Omega)\Big)^\frac{1}{p}.
\]
The reverse inequality simply follows from the fact that $C^\infty_0(\Omega\setminus E)\subset C^\infty_0(\Omega)$, thus we get the conclusion when $\Omega$ is bounded.
\par
In order to remove the last assumption, we consider the sets $\Omega_R=\Omega\cap B_R(0)$. For $R$ large enough, this is a non-empty open bounded set and $E\Subset \Omega_R$ as well. We thus have
\[
\lambda^s_p(\Omega_R\setminus E)=\lambda^s_p(\Omega_R).
\]
By taking the limit\footnote{Such a limit exists by monotonicity.} as $R$ goes to $+\infty$, we get the desired conclusion in the general case as well.
\end{proof}
The previous result giving the link between the Poincar\'e constant and sets with null capacity {\it does not hold true} in the interpolation space $\mathcal{X}^{s,p}_0(\Omega)$. Indeed, we have the following result, which shows that the interpolation Poincar\'e constant is sensitive to removing sets with null $(s,p)-$capacity. 
\begin{lm}
\label{lm:unnatural}
Let $1<p<N$ and $0<s<1$. Let $\Omega\subset\mathbb{R}^N$ be an open set and $E\Subset\Omega$ a compact set such that
\[
\mathrm{int\,cap}_{s,p}(E)=0<\mathrm{cap}_p(E).
\]
Then we have
\[
\Lambda^s_{p}(\Omega\setminus E)>\Lambda^s_p(\Omega).
\]
\end{lm}
\begin{proof}
By Corollary \ref{coro:uguali}, we know that 
\[
\Lambda^s_{p}(\Omega\setminus E)=\Big(\lambda^1_p(\Omega\setminus E)\Big)^s \qquad \mbox{ and }\qquad \Lambda^s_p(\Omega)=\Big(\lambda^1_p(\Omega)\Big)^s.
\]
It is now sufficient to use that $\lambda^1_p(\Omega\setminus E)>\lambda^1_p(\Omega)$, as a consequence of the fact that $E$ has positive $p-$capacity.
\end{proof}
\begin{rem}
\label{rem:flat}
As an explicit example of the previous situation, we can take $s\,p<1$ and the $(N-1)-$dimensional set 
\[
F=[-a,a]^{N-1}\times\{0\}.
\]
Observe that $\mathrm{cap}_p(F)>0$ by \cite[Theorem 4, page 156]{EG}. On the other hand, we have
\[
\mathrm{int\,cap}_{s,p}(F)=0.
\]
Indeed, we set
\[
F_\varepsilon=\{x\in \mathbb{R}^N \, :\, \mathrm{dist}(x,F)<\varepsilon\}.
\]
We then take the usual sequence of Friedrichs mollifiers $\{\varrho_\varepsilon\}_{\varepsilon>0}\subset C^\infty_0(\mathbb{R}^N)$ and define
\[
\varphi_{\varepsilon}=1_{F_\varepsilon}\ast \varrho_\varepsilon\in C^\infty_0(\mathbb{R}^N).
\] 
Observe that by construction we have
\[
\varphi_\varepsilon\equiv 1 \mbox{ on } F_\varepsilon\qquad \mbox{ and }\qquad \varphi_\varepsilon\equiv 0\mbox{ on } \mathbb{R}^N\setminus F_{2\,\varepsilon}.
\]
By definition of $(s,p)-$capacity and using the interpolation estimate \eqref{mah!}, we get
\[
\begin{split}
\mathrm{cap}_{s,p}(F)&\le [\varphi_\varepsilon]_{W^{s,p}(\mathbb{R}^N)}^p\\
&\le C\,\left(\int_{Q} |\varphi_\varepsilon|^p\,dx\right)^{1-s}\,\left(\int_Q |\nabla \varphi_\varepsilon|^p\,dx\right)^s\\
&\le C\, \left(\int_{Q} |1_{F_\varepsilon}|^p\,dx\right)^{1-s}\,\left(\int_Q |1_{F_\varepsilon}|^p\,dx\right)^s \left(\int_Q |\nabla \varrho_\varepsilon|\,dx\right)^{s\,p}\\
&\le C\,|F_\varepsilon|\,\varepsilon^{-s\,p}\le C\,\varepsilon^{1-s\,p}.
\end{split}
\]
We then observe that the last quantity goes to $0$ as $\varepsilon$ goes to $0$, thanks to the fact that $s\,p<1$.
By Corollary \ref{coro:zerocap}, we have 
\[
\mathrm{int\, cap}_{s,p}(F)=\mathrm{cap}_{s,p}(F)=0.
\]
as desired. 
\end{rem}

\section{Double-sided estimates for Poincar\'e constants}
\label{sec:poincare}

We already observed that for an open set $\Omega\subset\mathbb{R}^N$ we have
\[
s\,(1-s)\,\Lambda^s_{p}(\Omega)=\Big(\lambda^1_p(\Omega)\Big)^s.
\]
We now want to compare $\lambda^1_p$ with the sharp Poincar\'e constant for the embedding $\mathcal{D}^{s,p}_0(\Omega)\hookrightarrow L^p(\Omega)$.
\begin{thm}
 \label{prop:doubleside}
Let $1<p<\infty$ and $0<s<1$. Let $\Omega\subset\mathbb{R}^N$ be an open set, then
\begin{equation}
\label{oneside}
s\,(1-s)\,\lambda^s_{p}(\Omega)\le 2^{p\,(1-s)}\,N\,\omega_N\,\Big(\lambda^1_{p}(\Omega)\Big)^s.
\end{equation}
If in addition: 
\begin{itemize}
\item $\Omega\subset\mathbb{R}^N$ is bounded with Lipschitz boundary, then we also have the reverse inequality
\begin{equation}
\label{twoside}
\frac{1}{C_1}\,\Big(\lambda^1_{p}(\Omega)\Big)^s\le s\,(1-s)\,\lambda^s_{p}(\Omega),
\end{equation}
where $C_1>0$ is the same constant as in Theorem \ref{thm:finally};
\vskip.2cm
\item $\Omega\subset\mathbb{R}^N$ is convex, then we also have the reverse inequality
\begin{equation}
\label{twosideconv}
\frac{1}{C_2}\,\Big(\lambda^1_{p}(\Omega)\Big)^s\le s\,(1-s)\,\lambda^s_{p}(\Omega),
\end{equation}
where $C_2$ is the universal constant given by 
\[
C_2=\frac{\Big(\lambda^1_p(B_1(0))\Big)^s}{\mathcal{C}},
\]
and $\mathcal{C}=\mathcal{C}(N,p)>0$ is the same constant as in the Hardy inequality for $\mathcal{D}^{s,p}_0(\Omega)$ (see \cite[Theorem 1.1]{BC}).
\end{itemize}
\end{thm}
\begin{proof}
The first inequality \eqref{oneside} is a direct consequence of the interpolation inequality \eqref{mah!}. Indeed, by using the definition of $\lambda^s_{p}(\Omega)$, we obtain from this inequality
\[
s\,(1-s)\,\lambda^s_{p}(\Omega)\,\|u\|^p_{L^p(\Omega)}\le C\, \|u\|_{L^p(\Omega)}^{(1-s)\,p}\,\|\nabla u\|_{L^p(\Omega)}^{s\,p},
\]
for every $u\in C^\infty_0(\Omega)$. By simplifying the factor $\|u\|^p_{L^p(\Omega)}$ on both sides and taking the infimum over $C^\infty_0(\Omega)$, we get the claimed inequality.
\vskip.2cm\noindent
In order to prove \eqref{twoside}, for every $\varepsilon>0$ we take $\varphi\in C^\infty_0(\Omega)$ such that 
\[
\frac{\|\varphi\|_{\mathcal{D}^{s,p}_0(\Omega)}^p}{\|\varphi\|^p_{L^p(\Omega)}}<\lambda^s_p(\Omega)+\varepsilon,
\]
then we use Theorem \ref{thm:finally} to infer
\[
\frac{1}{C_1}\,\frac{\|\varphi\|_{\mathcal{X}^{s,p}_0(\Omega)}^p}{\|\varphi\|^p_{L^p(\Omega)}}<\lambda^s_p(\Omega)+\varepsilon.
\]
This in turn implies 
\[
\frac{1}{C_1}\,\Lambda^s_p(\Omega)\le \lambda^s_p(\Omega),
\]
by arbitrariness of $\varepsilon>0$. A further application of Corollary \ref{coro:uguali} leads to the desired conclusion.
\vskip.2cm\noindent
Finally, if $\Omega\subset\mathbb{R}^N$ is convex, we can proceed in a different way. We first observe that we can always suppose that $R_\Omega<+\infty$, otherwise both $\lambda^1_p(\Omega)$ and $\lambda^s_p(\Omega)$ vanish and there is nothing to prove.
Then \eqref{twosideconv} comes by joining the simple estimate
\[
\lambda^1_{p}(\Omega)\le \frac{\lambda^1_p(B_1(0))}{R_\Omega^p},
\]
which follows from the monotonicity and scaling properties of $\lambda^1_p$, and the estimate of \cite[Corollary 5.1]{BC}, i.e.
\[
s\,(1-s)\,\lambda^s_p(\Omega)\ge \frac{\mathcal{C}}{R_\Omega^{s\,p}}.
\]
The latter is a consequence of the Hardy inequality in convex sets for $\mathcal{D}^{s,p}_0$.
\end{proof}

\begin{rem}
For $p=2$, the double-sided estimate of Theorem \ref{prop:doubleside} is contained in \cite[Theorem 4.5]{CS}. The proof in \cite{CS} relies on probabilistic techniques. In \cite{CS} the result is proved by assuming that $\Omega$ verifies a uniform exterior cone condition. 
\end{rem}

\begin{rem}
\label{rem:poinpoin}
Inequality \eqref{twoside} can not hold for a general open set $\Omega\subset\mathbb{R}^N$, with a constant independent of $\Omega$. Indeed, one can construct a sequence $\{\Omega_n\}_{n\in\mathbb{N}}\subset\mathbb{R}^N$ such that
\[
\lim_{n\to\infty}\frac{\Big(\lambda^1_{p}(\Omega_n)\Big)^s}{\lambda^s_{p}(\Omega_n)}=+\infty,\qquad \mbox{ for } 1<p<\infty \mbox{ and } s< \frac{1}{p},
\]
see Lemma \ref{lm:bas} below.
\end{rem}

\appendix

\section{An example} 
\label{sec:example} 
 
In this section, we construct a sequence of open bounded sets $\{\Omega_n\}_{n\in\mathbb{N}}\subset\mathbb{R}^N$ with rough boundaries and fixed diameter, such that we have
\begin{equation}
\label{counter2}
\lim_{n\to\infty}\frac{\Big(\lambda^1_p(\Omega_n)\Big)^s}{\lambda^s_{p}(\Omega_n)}=+\infty,\qquad \mbox{ for } 1<p<\infty \mbox{ and } s< \frac{1}{p}.
\end{equation}
The sets $\Omega_n$ are obtained by removing from an $N-$dimensional cube an increasing array of regular $(N-1)-$dimensional cracks.
\vskip.2cm\noindent
For $N\ge 1$, we set\footnote{For $N=1$, the set $F$ simply coincides with the point $\{0\}$.}
\[
Q=\left[-\frac{1}{2},\frac{1}{2}\right]^N\qquad \mbox{ and }\qquad F=\left[-\frac{1}{4},\frac{1}{4}\right]^{N-1}\times\{0\}.
\]
For every $n\in \mathbb{N}$, we also define
\[
\mathbb{Z}^N_n=\Big\{z=(z_1,\dots,z_N)\in\mathbb{Z}^N\, :\, \max\{|z_1|,\dots,|z_N|\}\le n\Big\}.
\]
Finally, we consider the sets
\[
\Omega=Q\setminus F,\qquad \widetilde\Omega_n=\bigcup_{z\in\mathbb{Z}^N_n} \Big(\Omega+z\Big)=\left[-n-\frac{1}{2},n+\frac{1}{2}\right]^N\setminus \bigcup_{z\in\mathbb{Z}^N_n} (F+z),
\]
and
\[
E=\bigcup_{n\in\mathbb{N}} \widetilde\Omega_n=\mathbb{R}^N\setminus \bigcup_{z\in\mathbb{Z}^N} (F+z).
\]
Then \eqref{counter2} is a consequence of the next result.
\begin{figure}
\includegraphics[scale=.35]{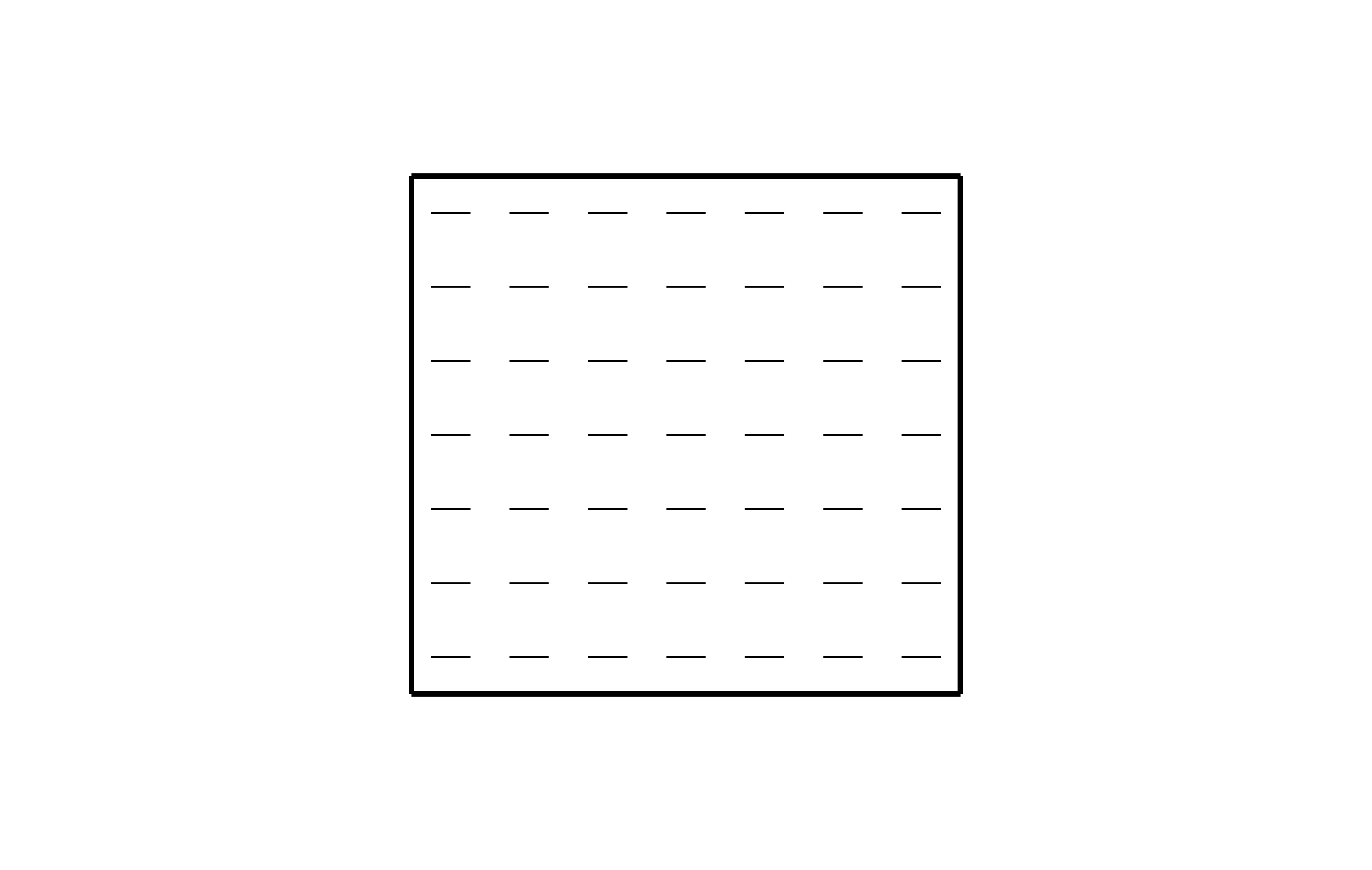}
\caption{The set $\Omega_n$ in dimension $N=2$, for $n=3$.}
\end{figure}
\begin{lm}
\label{lm:bas}
With the notation above, for $1<p<\infty$ and $s<1/p$ we have
\begin{equation}
\label{bottom}
\lambda^1_{p}(\widetilde\Omega_n)\ge C=C(N,p,F)>0,\qquad \mbox{ for every }n\in\mathbb{N},
\end{equation}
and
\begin{equation}
\label{up}
\lim_{n\to\infty} \lambda^s_{p}(\widetilde\Omega_n)=0.
\end{equation}
In particular, the new sequence of rescaled sets $\{\Omega_n\}_{n\in\mathbb{N}}\subset\mathbb{R}^N$ defined by
\[
\Omega_n=|\widetilde\Omega_n|^{-\frac{1}{N}}\,\widetilde\Omega_n=\left[-\frac{1}{2},\frac{1}{2}\right]^N\setminus \bigcup_{z\in\mathbb{Z}^N_n} \frac{(F+z)}{2\,n+1},
\]
is such that
\[
\mathrm{diam}(\Omega_n)=\sqrt{N}, \mbox{ for every }n\in\mathbb{N}\qquad \mbox{ and }\qquad\lim_{n\to\infty}\frac{\left(\lambda^1_{p}(\Omega_n)\right)^s}{\lambda^s_{p}(\Omega_n)}=+\infty.
\]
\end{lm}
\begin{proof}
We divide the proof in two parts, for ease of readability. Of course, it is enough to prove \eqref{bottom} and \eqref{up}. Indeed, the last statement is a straightforward consequence of these facts and of the scaling properties of the diameter and of the Poincar\'e constants.
\vskip.2cm\noindent
{\bf Proof of \eqref{bottom}}. For $1<p<\infty$ we define
\[
\mu_p(Q;F)=\min_{u\in W^{1,p}(Q)\setminus\{0\}}\left\{\frac{\displaystyle\int_{Q}|\nabla u|^p\,dx}{\displaystyle \int_{Q}|u|^p\,dx}\, :\, u=0 \mbox{ on } F\right\}.
\]
We first observe that $F$ is a compact set with positive $(N-1)-$dimensional Hausdorff measure, thus by \cite[Theorem 4, page 156]{EG} we have 
\[
\mathrm{cap}_p(F;Q)=\inf_{u\in C^\infty_0(Q)}\left\{\int_Q |\nabla u|^p\,dx\, :\, u\ge 1_F\right\}>0,\qquad \mbox{ for every } 1<p<\infty. 
\]
We can thus infer existence of a constant  $C=C(N,p,F)>0$ such that
\[
\frac{1}{C}\,\int_{Q} |u|^p\,dx\le \int_Q |\nabla u|^p\,dx,\qquad \mbox{ for every }u\in W^{1,p}(Q)\mbox{ such that } u=0 \mbox{ on } F,
\]
see \cite[Theorem 10.1.2]{Ma}.
This shows that $\mu_p(Q;F)>0$.
\par
For every $\varepsilon>0$, we consider $u_{\varepsilon}\in C_0^\infty(E)\setminus\{0\}$ such that
\[
\lambda^1_{p}(E)+\varepsilon> \frac{\displaystyle\int_{E}|\nabla u_\varepsilon|^p\,dx}{\displaystyle \int_{E}|u_\varepsilon|^p\,dx}.
\]
We now observe that for every $z\in\mathbb{Z}^N$, there holds
\[
\int_{Q+z} |\nabla u_\varepsilon|^p\,dx\ge \mu(Q,F)\,\int_{Q+z} |u_\varepsilon|^p\,dx,
\]
thanks to the fact that $u_n$ vanishes on (the relevant translated copy of) $F$ and to the fact that $\mu_p(Q,F)=\mu_p(Q+z,F+z)$. By using this information, we get
\[
\begin{split}
\int_{E}|\nabla u_\varepsilon|^p\,dx&=\sum_{z\in\mathbb{Z}^N} \int_{Q+z} |\nabla u_\varepsilon|^p\,dx\\
&\ge \mu_p(Q,F)\,\sum_{z\in\mathbb{Z}^N}\int_{Q+z} |u_\varepsilon|^p\,dx=\mu_p(Q,F)\, \int_{E} |u_\varepsilon|^p\,dx.
\end{split}
\]
By recalling the choice of $u_\varepsilon$, we then get
\[
\lambda^1_{p}(E)+\varepsilon\ge \mu_p(Q;F).
\]
Thanks to the arbitrariness of $\varepsilon>0$ and to the fact that $\widetilde\Omega_n\subset E$, this finally gives
\begin{equation}
\label{E}
\lambda^1_{p}(\widetilde\Omega_n)\ge \lambda^1_{p}(E)\ge \mu_p(Q;F),\qquad \mbox{ for every }n\in\mathbb{N},
\end{equation}
as desired.
\vskip.2cm\noindent
{\bf Proof of \eqref{up}}. We recall that
\[
\widetilde\Omega_n=\bigcup_{z\in\mathbb{Z}^N_n} \Big(\Omega+z\Big)=\left[-n-\frac{1}{2},n+\frac{1}{2}\right]^N\setminus \bigcup_{z\in\mathbb{Z}^N_n} (F+z),
\]
and that each $(N-1)-$dimensional set $F+z$ has null $(s,p)-$capacity, thanks to Remark \ref{rem:flat}. By using Proposition \ref{prop:submodular}, we also obtain
\[
\mathrm{cap}_{s,p}\left(\bigcup_{z\in\mathbb{Z}^N_n} (F+z)\right)=0.
\]
Then by Proposition \ref{prop:capdim}, we get
\[
\lambda^s_p(\widetilde\Omega_n)=\lambda^s_p\left(\left[-n-\frac{1}{2},n+\frac{1}{2}\right]^N\right)=(2\,n+1)^{-s\,p}\,\lambda^{s}_p(Q).
\]
This is turn gives the desired conclusion \eqref{up}.
\end{proof}

\section{One-dimensional Hardy inequality}

We used the following general form of the one-dimensional Hardy inequality (the classical case corresponds to $\alpha=p-1$ below). This can be found for example in \cite{Ma}. For the sake of completeness, we give a sketch of a proof based on {\it Picone's inequality}.\footnote{For $u,v$ differentiable functions with $v\ge 0$ and $u>0$, we have the pointwise inequality
\[
|u'|^{p-2}\,u'\,\left(\frac{v^p}{u^{p-1}}\right)'\le |v'|^p.
\]}
\begin{lm}
\label{lm:hardy1D}
Let $1<p<\infty$ and $\alpha>0$. For every $f\in C^\infty_0((0,T])$ we have
\begin{equation}
\label{hardy}
\left(\frac{\alpha}{p}\right)^p\,\int_0^T \frac{|f(t)|^p}{t^\alpha}\,\frac{dt}{t}\le \int_0^T \frac{|f'(t)|^p}{t^\alpha}\,t^p\,\frac{dt}{t}.
\end{equation}
\end{lm}
\begin{proof}
We take $0<\beta<\alpha/(p-1)$ and consider the function $\varphi(t)=t^\beta$. Observe that this solves
\[
\begin{split}
-\big(|\varphi'(t)|^{p-2}\,\varphi'(t)\,t^{p-\alpha-1}\big)'&=\beta^{p-1}\,(\alpha-\beta\,(p-1))\,t^{\beta\,(p-1)-\alpha-1}\\
&=\beta^{p-1}\,(\alpha-\beta\,(p-1))\,t^{-\alpha-1}\,\varphi(t)^{p-1}.
\end{split}
\]
Thus, for every $\psi\in C_0^\infty((0,T])$ we have the weak formulation
\[
\beta^{p-1}\,(\alpha-\beta\,(p-1))\,\int_0^T \frac{\varphi^{p-1}}{t^\alpha}\,\psi\,\frac{dt}{t}=\int_0^T \frac{|\varphi'|^{p-2}\,\varphi'}{t^\alpha}\,\psi'\,t^p\,\frac{dt}{t}.
\]
We take $\varepsilon>0$ and $f\in C^\infty_0((0,T])$ non-negative, we insert the test function 
\[
\psi=\frac{f^p}{(\varepsilon+\varphi)^{p-1}},
\]
in the previous integral identity. By using Picone's inequality, we then obtain
\[
\begin{split}
\beta^{p-1}\,(\alpha-\beta\,(p-1))\,\int_0^T \frac{\varphi^{p-1}}{(\varepsilon+\varphi)^{p-1}}\,\frac{f^p}{t^\alpha}\,\frac{dt}{t}&=\int_0^T \frac{|\varphi'|^{p-2}\,\varphi'}{t^\alpha}\,\left(\frac{f^p}{(\varepsilon+\varphi)^{p-1}}\right)'\,t^p\,\frac{dt}{t}\\
&\le \int_0^T \frac{|f'(t)|^p}{t^\alpha}\,t^p\,\frac{dt}{t}.
\end{split}
\]
If we take the limit as $\varepsilon$ goes to $0$, by Fatou's Lemma we get
\[
\beta^{p-1}\,(\alpha-\beta\,(p-1))\,\int_0^T \frac{f^p}{t^\alpha}\,\frac{dt}{t}\le \int_0^T \frac{|f'(t)|^p}{t^\alpha}\,t^p\,\frac{dt}{t}.
\]
The previous inequality holds true for every $0<\beta<\alpha/(p-1)$ and $\beta^{p-1}\,(\alpha-\beta\,(p-1))$ is maximal for $\beta=\alpha/p$. This concludes the proof.
\end{proof}

\section{A geometric lemma}

When comparing the norms of $\mathcal{X}^{s,p}_0(\Omega)$ and $\mathcal{D}^{s,p}_0(\Omega)$ for a convex set, we used the following geometric result. We recall that 
\[
R_\Omega=\sup_{x\in\Omega} \mathrm{dist}(x,\partial\Omega),
\]
is the {\it inradius of $\Omega$}, i.e. the radius of the largest ball inscribed in $\Omega$.
\begin{lm}
\label{lm:convex}
Let $\Omega\subset\mathbb{R}^N$ be an open convex set such that $R_\Omega<+\infty$. Let $x_0\in\Omega$ be a point such that
\[
\mathrm{dist}(x_0,\partial\Omega)=R_\Omega. 
\]
Then for every $0<t<1$ we have
\[
\mathrm{dist}\big(x_0+t\,(\Omega-x_0),\partial\Omega\big)\ge (1-t)\,R_\Omega.
\]
\end{lm}
\begin{proof}
Without loss of generality, we can assume that $0\in \Omega$ and that $x_0=0$. Clearly, it is sufficient to prove that
\[
\mathrm{dist}\big(\partial(t\,\Omega),\partial\Omega\big)\ge (1-t)\,R_\Omega.
\]
Every point of $\partial(t\,\Omega)$ is of the form $t\,z$, with $z\in\partial\Omega$. We now take the cone $C_z$, obtained as the convex envelope of $B_{R_\Omega}(0)$ and the point $z$. 
By convexity of $\Omega$, we have of course $C_z\subset \Omega$.  We thus obtain
\begin{equation}
\label{dist}
\mathrm{dist}(t\,z,\partial\Omega)\ge \mathrm{dist}(t\,z,\partial C_z).
\end{equation}
We now distinguish two cases:
\begin{itemize}
\item[({\it i)}] $|z|=R_\Omega$;
\vskip.2cm
\item[{\it (ii)}] $|z|>R_\Omega$.
\end{itemize}
When alternative {\it i)} occurs, then $C_z=B_{R_\Omega}(0)$ and thus
\[
\mathrm{dist}(t\,z,\partial C_z)=\mathrm{dist}(t\,z,B_{R_\Omega}(0))=|t\,z-z|=(1-t)\,|z|=(1-t)\,R_\Omega.
\]
By using this in \eqref{dist}, we get the desired estimate.
\par
If on the contrary we are in case {\it ii)}, then by elementary geometric considerations we have 
\[
\frac{\mathrm{dist}(t\,z,\partial C_z)}{|t\,z-z|}=\frac{R_\Omega}{|z|},
\]
see Figure \ref{fig:convex}.
This gives again the desired conclusion.
\end{proof}
\begin{figure}
\includegraphics[scale=.4]{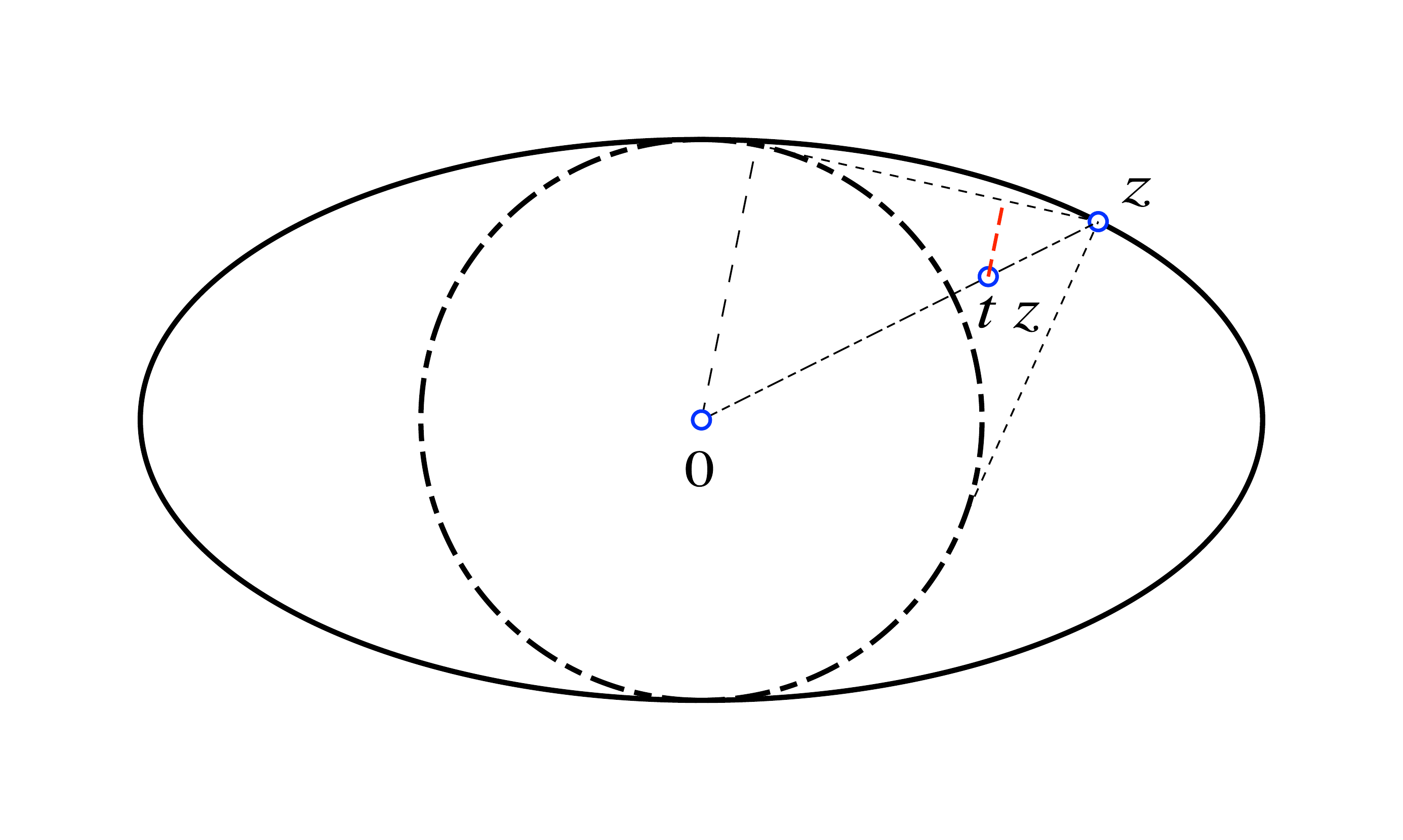}
\caption{The case {\it (ii)} in the proof of Lemma \ref{lm:convex}. Colored in red, the distance of $t\,z$ from $\partial C_z$.}
\label{fig:convex}
\end{figure}

\end{document}